\documentclass[11pt]{article}

\usepackage{amsmath,amssymb,amsthm,amsfonts,bm,geometry,graphicx,hyperref}
\numberwithin{equation}{section}
\newcommand{\figureprefix}{}
\IfFileExists{nn.pdf}{}{%
  \renewcommand{\figureprefix}{arxiv/}%
}

\usepackage{subcaption}

\newtheorem{thm}{Theorem}[section]

\newcommand{\R}{\mathbb{R}}

\newcommand{\vect}[1]{\boldsymbol{#1}}
\newcommand{\half}{\frac{1}{2}}
\newcommand{\diag}{\operatorname{diag}}
\makeatletter
\renewcommand{\@maketitle}{%
  \newpage
  \null
  \vskip 2em%
  \begin{center}%
  {\Large \@title \par}%
  \vskip 1.5em%
  {\large
    \lineskip .5em%
    \begin{tabular}[t]{c}%
      \@author
    \end{tabular}\par}%
  \vskip 1em%
  {\large \@date}%
  \end{center}%
  \par
  \vskip 1.5em}
\makeatother

\title{Machine learning moment closure models for the radiative transfer equation IV: enforcing symmetrizable hyperbolicity in two dimensions}

\author{Juntao Huang\thanks{Department of Mathematical Sciences, University of Delaware, Newark, DE, 19716, USA. (\texttt{huangjt@udel.edu}). Research was partially supported by the National Science Foundation (DMS-2309655 and DMS-2618114). Part of this research was performed while the author was visiting the Institute for Pure and Applied Mathematics (IPAM) at the University of California, Los Angeles, which is supported by the National Science Foundation (DMS-1925919).}}

\date{\today}

\begin{document}
\maketitle

\begin{abstract}
\normalsize

This is our fourth work in the series on machine learning (ML) moment closure models for the radiative transfer equation (RTE). 
In the first three papers of this series, we considered the RTE in slab geometry in 1D1V (i.e. one dimension in physical space and one dimension in angular space), and introduced a gradient-based ML moment closure \cite{huang2022mlmc1}, then enforced the hyperbolicity through a symmetrizer \cite{huang2021mlmc2}, or together with physical characteristic speeds by learning the eigenvalues of the Jacobian matrix \cite{huang2022mlmc3}.

Here, we extend our framework to the RTE in 2D2V (i.e. two dimensions in physical space and two dimensions in angular space). The main idea is to preserve the leading part of the classical $P_N$ model and modify only the highest-order block row. By analyzing the structural properties of the $P_N$ model, we show that its coefficient matrices are symmetric and admit a block-tridiagonal structure. Then we use this property to introduce a block-diagonal symmetrizer for the ML moment model and derive explicit algebraic conditions on the closure blocks which guarantee the symmetrizable hyperbolicity of the resulting ML system. These conditions lead to a natural parametrization of the closure in terms of a symmetric positive definite matrix together with symmetric closure blocks, which can be learned from data while automatically enforcing symmetrizable hyperbolicity by construction. The numerical results show that the proposed framework improves upon the classical $P_N$ model while maintaining hyperbolicity.
\end{abstract}


\section{Introduction}

The radiative transfer equation (RTE) describes particle propagation and interaction with a background medium. It has been widely applied in many fields of science and engineering, including astrophysics \cite{pomraning2005equations}, heat transfer \cite{koch2004evaluation}, and optical imaging \cite{klose2002optical}.
Because the unknown function in the RTE (also called the specific intensity) depends on time, physical space, and angular variables, direct simulation of the kinetic equation remains computationally demanding due to the curse of dimensionality \cite{huang2024adaptive,peng2021reduced,einkemmer2025review}.

As one approach to reduce the computational cost, moment methods track a finite number of moments of the specific intensity. However, the evolution of the moments always depends on the next higher-order moments, leading to the moment closure problem \cite{grad1949kinetic}. Many closure models have been developed, including the $P_N$ model \cite{chandrasekhar1944radiative}, variable Eddington factor models \cite{levermore1984relating,murchikova2017analytic}, entropy-based $M_N$ models \cite{hauck2011high,alldredge2012high,alldredge2014adaptive}, positive $P_N$ models \cite{hauck2010positive}, filtered $P_N$ models \cite{mcclarren2010robust,laboure2016implicit,radice2013new}, $B_2$ models \cite{alldredge2016approximating}, and $MP_N$ models \cite{fan2020nonlinear,fan2020nonlinear2}, to name a few.

Thanks to the rapid development of machine learning (ML) and data-driven modeling \cite{lecun2015deep,brunton2016discovering,raissi2019physics,han2018solving}, ML-based moment closure has become an active research direction. In \cite{han2019uniformly}, the authors introduced a framework that learns a moment closure model for the BGK equation. Other works include stable ML closures based on the conservation-dissipation formalism \cite{huang2020learning,zhu2015conservation}, nonlocal closures for the Vlasov-Poisson system \cite{bois2020neural}, neural network approximations of Landau fluid closures \cite{ma2020machine,wang2020deep,maulik2020neural}, and WSINDy closures for thermal radiation transport \cite{messenger2025learning}. In parallel, ML has also been used to solve or accelerate high-dimensional kinetic equations directly, including BGK models \cite{lou2020physics}, phonon Boltzmann equations \cite{li2021physics}, the RTE \cite{mishra2020physics}, Boltzmann collision operators \cite{xiao2020using,miller2022neural}, and learning collision operators from molecular dynamics \cite{zhao2025data,zhao2026molecular}. More recently, structure-preserving surrogates for entropy-based closures have been developed using convex splines and input convex neural networks \cite{hauck2021entropy,hauck2021boltzmann,amos2017input,schotthofer2025structure}, providing a different route to combining learning with mathematical structure. These works have shown the promise of ML closures for kinetic equations, but they also indicate that careful design is needed to ensure that the learned closure preserves the structural properties of the underlying PDE. For example, if the learned closure breaks hyperbolicity, the resulting system can lose long-time stability and generate unphysical solutions, even if it fits the training data well. Hence, it is important to incorporate structure preservation into the ML closure design.

In our previous work, we proposed a structure-preserving ML moment closure framework for the RTE in slab geometry \cite{huang2022mlmc1,huang2021mlmc2,huang2022mlmc3}. In the first work \cite{huang2022mlmc1}, motivated by the exact free-streaming closure, we proposed to learn the gradient of the unclosed moment rather than the moment itself \cite{huang2022mlmc1}. This gradient-based ansatz also provides a natural output normalization. In the second work \cite{huang2021mlmc2}, we enforced the hyperbolicity by constructing a symmetrizer for the ML closure system. In the third work \cite{huang2022mlmc3}, we made use of the lower Hessenberg structure of the coefficient matrix to approximate the eigenvalues by neural networks with additional control of characteristic speeds. This framework was extended to the BGK equations in \cite{christlieb2025hyperbolic} and the linear Boltzmann equation with uncertainties in \cite{huang2026machine}.
Together, these works showed that accurate ML closures can be built only when the neural network architecture is carefully designed to preserve the structural properties of the PDE.

The purpose of this paper is to extend our framework from the radiative transfer equation (RTE) in slab geometry, namely the 1D1V setting (one physical dimension and one angular dimension), to the RTE in the 2D2V setting (two physical dimensions and two angular dimensions). This extension is highly nontrivial. In the 1D case, the hyperbolicity of the moment system is determined by the eigenstructure of a single coefficient matrix. In 2D, however, one must require that any linear combination of the coefficient matrices in the $x$- and $y$-directions is diagonalizable over $\mathbb{R}$, which is a substantially stronger and more delicate condition.
In addition, the algebraic structures of the coefficient matrices in the 2D system become much more complicated than in 1D. In the 1D setting, moments are ordered by the degree of the Legendre polynomial basis, and the coefficient matrix has a lower Hessenberg form, making it possible to construct closures by modifying only the last row. In 2D, however, moments are organized by both degree and order of the spherical harmonics basis, and the coefficient matrices exhibit a block structure indexed by degree. The resulting coupling pattern is much more intricate: each moment interacts with multiple neighboring moments within the same degree as well as across adjacent degrees. As a consequence, the closure strategies developed in 1D do not directly extend to 2D.

In this work, we address these challenges by first identifying the structural properties of the real-valued $P_N$ system in 2D. We show that the coefficient matrices of the $P_N$ model are symmetric and block tridiagonal, with blocks indexed by degree. This structure makes it possible to preserve the leading part of the $P_N$ system and modify only the highest-order block row, which forms the foundation of our ML closure construction. We then introduce a block-diagonal symmetrizer and derive explicit algebraic conditions on the closure blocks that ensure symmetrizable hyperbolicity of the resulting system. This perspective leads to a natural parametrization of the ML closure in terms of a symmetric positive definite matrix together with symmetric closure blocks, which can be represented by neural networks while automatically enforcing the required structural constraints by construction.
Our ML moment closure is trained through a residual-based loss function for the highest-order moment equation. Numerical experiments are conducted on a series of tasks which clearly show that the proposed method can improve on the classical $P_N$ model when trained on kinetic reference data. 

The remainder of the paper is organized as follows. In Section~\ref{sec:pn-2d}, we derive the real-valued $P_N$ system in 2D and analyze its structural properties. In Section~\ref{sec:hyperbolic-ml-closure}, we introduce the hyperbolic ML closure framework with its constrained parametrization. Section~\ref{sec:numerical-experiments} reports numerical experiments. Finally, we conclude in Section~\ref{sec:conclusion} with a discussion of future work.

\section{The RTE and the \texorpdfstring{$P_N$}{PN} model in 2D}\label{sec:pn-2d}

In this section, we begin from the RTE and derive the two-dimensional
$P_N$ system. Then we analyze the structural properties of the $P_N$ system, which will later guide the ML closure construction. In particular, we will see that the coefficient matrices in the $P_N$ system are symmetric and admit a block-tridiagonal structure after grouping the moments by polynomial degree.

We consider the RTE in two spatial dimensions with isotropic scattering:
\begin{equation}\label{eq:rte}
\partial_t \psi + \Omega_x\partial_x\psi + \Omega_y\partial_y\psi
= \sigma_s\left(\frac{1}{4\pi}\int_{\mathbb S^2}\psi(t,x,y,\Omega')\,d\Omega' - \psi\right) - \sigma_a\psi,
\end{equation}
for $(t,x,y)\in \mathbb R_+\times\mathbb R^2$ and $\Omega=(\Omega_x,\Omega_y,\Omega_z)\in \mathbb S^2$. The angular variable $\Omega$ can also be parameterized by
\begin{equation}
    \Omega=(\sqrt{1-\mu^2}\cos\varphi,\,\sqrt{1-\mu^2}\sin\varphi,\,\mu), \qquad \mu\in[-1,1],\ \varphi\in[0,2\pi).
\end{equation}
Here $\psi = \psi(t,x,y,\Omega)$ is the specific intensity, $\sigma_s = \sigma_s(x,y) \ge 0$ is the scattering cross section, and $\sigma_a = \sigma_a(x,y) \ge 0$ is the absorption cross section. The collision operator on the right hand side is isotropic, since the scattering redistributes particles uniformly in all directions.

We assume the problem is planar, so that the coefficients are independent of the third spatial variable and the solution is even with respect to $\mu=\Omega_z$:
\[
\psi(t,x,y,\mu,\varphi)=\psi(t,x,y,-\mu,\varphi).
\]
This symmetry is preserved by \eqref{eq:rte}, since both transport coefficients $\Omega_x$, $\Omega_y$ and the isotropic collision operator are even in $\mu$.

To derive the $P_N$ system, we first introduce a set of real-valued orthonormal basis of $L^2(\mathbb S^2)$. Define the real-valued spherical harmonics by
\begin{align}
\Phi_l^0(\mu,\varphi) &:= N_l^0 P_l(\mu), \qquad l\ge 0, \\
\Phi_l^{m,\mathrm c}(\mu,\varphi) &:= \sqrt{2}\,N_l^m P_l^m(\mu)\cos(m\varphi), \qquad 1\le m\le l, \\
\Phi_l^{m,\mathrm s}(\mu,\varphi) &:= \sqrt{2}\,N_l^m P_l^m(\mu)\sin(m\varphi), \qquad 1\le m\le l.
\end{align}
Here $P_l$ and $P_l^m$ denote the Legendre and associated Legendre polynomials and the constants $N_l^m := \sqrt{\frac{2l+1}{4\pi}\frac{(l-m)!}{(l+m)!}}$. By the orthogonality of the (associated) Legendre polynomials in $\mu$ and the orthogonality of $\{1, \cos(m\varphi), \sin(m\varphi)\}$ in $\varphi$, the family
\begin{equation}
    \bigl\{\Phi_l^0,\ \Phi_l^{m,\mathrm c},\ \Phi_l^{m,\mathrm s}: l\ge 0,\ 1\le m\le l\bigr\}
\end{equation}
forms an orthonormal basis of $L^2(\mathbb S^2)$; see the details in Theorem \ref{thm:real-sph-harm-orthonormal} in the appendix. In addition, the parity relations $P_l(-\mu) = (-1)^l P_l(\mu)$ and $P_l^m(-\mu)=(-1)^{l+m}P_l^m(\mu)$ imply
\begin{equation}
\Phi_l^0(-\mu,\varphi)=(-1)^l\Phi_l^0(\mu,\varphi),
\qquad
\Phi_l^{m,\mathrm c/\mathrm s}(-\mu,\varphi)=(-1)^{l+m}\Phi_l^{m,\mathrm c/\mathrm s}(\mu,\varphi).
\end{equation}
Hence, for an even function $\psi$ of $\mu$, we introduce the truncated space
\begin{equation}
V_N := \mathrm{span}\Bigl(\{\Phi_l^0: 0\le l\le N,\ l\ \text{even}\}
\cup \{\Phi_l^{m,\mathrm c},\ \Phi_l^{m,\mathrm s}: 0\le l\le N,\ 1\le m\le l,\ l+m\ \text{even}\}\Bigr),
\end{equation}
and approximate $\psi$ by
\begin{equation}\label{eq:real-expansion}
\psi_N(t,x,y,\Omega)
=
\sum_{\substack{0\le l\le N \\ l\ \mathrm{even}}} R_l^0(t,x,y)\,\Phi_l^0(\Omega)
+
\sum_{l=1}^N\sum_{\substack{1\le m\le l \\ l+m\ \mathrm{even}}}
\Bigl(
R_l^m(t,x,y)\,\Phi_l^{m,\mathrm c}(\Omega)
+
I_l^m(t,x,y)\,\Phi_l^{m,\mathrm s}(\Omega)
\Bigr),
\end{equation}
where the coefficients are the moments of $\psi$ in the corresponding basis functions:
\begin{equation}
R_l^0 = \langle \psi,\Phi_l^0\rangle,
\qquad
R_l^m = \langle \psi,\Phi_l^{m,\mathrm c}\rangle,
\qquad
I_l^m = \langle \psi,\Phi_l^{m,\mathrm s}\rangle,
\end{equation}
with $\langle \cdot \rangle$ the inner product in $L^2(\mathbb S^2)$.
For each degree $l$, the number of moments is $l+1$. We order them as $\vect{u}_l\in\mathbb R^{l+1}$:
\begin{equation}
\vect{u}_l=
\begin{cases}
(R_l^0,\,R_l^2,\,I_l^2,\,\dots,\,R_l^l,\,I_l^l)^T, & l\ \text{even},\\[0.3em]
(R_l^1,\,I_l^1,\,R_l^3,\,I_l^3,\,\dots,\,R_l^l,\,I_l^l)^T, & l\ \text{odd},
\end{cases}
\end{equation}
and collect all the moments in
\begin{equation}
\vect{u}=(\vect{u}_0,\vect{u}_1,\dots,\vect{u}_N)^T\in\mathbb R^{\frac{1}{2}(N+1)(N+2)}.
\end{equation}

For $l\ge 0$, let $V_l\subset V_N$ denote the span of the basis functions of degree $l$. Using the recursion relations for associated Legendre polynomials together with the trigonometric identities,
one finds that multiplication by $\Omega_x=\sqrt{1-\mu^2}\cos\varphi$ and $\Omega_y=\sqrt{1-\mu^2}\sin\varphi$ changes the degree only from $l$ to $l\pm1$ and the order only from $m$ to $m\pm1$. Hence, we have the degree coupling property
\begin{equation}\label{eq:degree-coupling}
\Omega_x V_l \subset V_{l-1}\oplus V_{l+1},
\qquad
\Omega_y V_l \subset V_{l-1}\oplus V_{l+1}.
\end{equation}

Projecting \eqref{eq:rte} onto $V_N$ yields the real-valued $P_N$ system
\begin{equation}\label{eq:real-pn}
    \partial_t \vect{u} + \mathcal{A} \, \partial_x\vect{u} + \mathcal{B} \, \partial_y\vect{u} = \mathcal{Q}\vect{u},
\end{equation}
where the coefficient matrices $\mathcal{A}$ and $\mathcal{B}$ are generated by the transport terms $\Omega_x$ and $\Omega_y$: 
\begin{equation}\label{eq:coefficient-matrices-def}
(\mathcal{A})_{\alpha\beta}=\langle \Omega_x\Phi_\beta,\Phi_\alpha\rangle,
\qquad
(\mathcal{B})_{\alpha\beta}=\langle \Omega_y\Phi_\beta,\Phi_\alpha\rangle.
\end{equation}
Here $\mathcal{Q}=\mathcal{Q}(x,y)$ is the diagonal collision matrix derived from absorption and isotropic scattering on the right hand side of \eqref{eq:rte}:
\begin{equation}\label{eq:collision-matrix}
\mathcal{Q}=\diag(Q_0,Q_1,\dots,Q_N),
\qquad
Q_0=-\sigma_a,
\qquad
Q_l=-(\sigma_a+\sigma_s)I_{l+1}, \quad 1\le l\le N.
\end{equation}

By the definition of $\mathcal{A}$ and $\mathcal{B}$ in \eqref{eq:coefficient-matrices-def}, one can easily show that they are symmetric matrices. Moreover, by the degree coupling property in \eqref{eq:degree-coupling}, the degree ordering $\vect{u}=(\vect{u}_0,\dots,\vect{u}_N)^T$ gives the block tridiagonal form
\begin{equation}\label{eq:block-matrices}
\mathcal{A} = \half
\begin{pmatrix}
0 & A_0^T \\
A_0 & 0 & A_1^T \\
& A_1 & 0 & A_2^T \\
& & \ddots & \ddots & \ddots \\
& & & A_{N-2} & 0 & A_{N-1}^T \\
& & & & A_{N-1} & 0
\end{pmatrix},
\quad
\mathcal{B} = \half
\begin{pmatrix}
0 & B_0^T \\
B_0 & 0 & B_1^T \\
& B_1 & 0 & B_2^T \\
& & \ddots & \ddots & \ddots \\
& & & B_{N-2} & 0 & B_{N-1}^T \\
& & & & B_{N-1} & 0
\end{pmatrix},
\end{equation}
with
\begin{equation}
A_l,B_l\in\mathbb R^{(l+2)\times(l+1)}, \quad 0\le l\le N-1.
\end{equation}
Equivalently, for $1\le l\le N-1$ the moment of degree $l$ satisfies
\begin{equation}\label{eq:block-equation}
\partial_t\vect{u}_l +\half A_{l-1}\partial_x\vect{u}_{l-1} +\half A_{l}^T\partial_x\vect{u}_{l+1} +\half B_{l-1}\partial_y\vect{u}_{l-1} +\half B_{l}^T\partial_y\vect{u}_{l+1} = Q_l\vect{u}_l.
\end{equation}
This indicates that the evolution of the moments of degree $l$ only depends on the moments of degree $l-1$, $l$, and $l+1$. The explicit form of the blocks $A_l$ and $B_l$ can be derived from the recursion relations for associated Legendre polynomials. We omit the details here since they are not needed for the ML closure, and refer the readers to \cite{seibold2014starmap} for more details.

Therefore, the $P_N$ system \eqref{eq:real-pn} is hyperbolic with the coefficient matrices $\mathcal{A}$ and $\mathcal{B}$ in \eqref{eq:block-matrices} symmetric and block tridiagonal. This is the starting point for the ML closure construction in the next section.

\section{Hyperbolic ML moment closure}\label{sec:hyperbolic-ml-closure}

Our ML closure idea is to preserve the head of the real-valued $P_N$ system and modify only the last block row. This is equivalent to modifying only the evolution equation for the highest-degree moment $\vect{u}_N$ while keeping the equations for $\vect{u}_0,\dots,\vect{u}_{N-1}$ unchanged. The motivation is that the evolution equations of the lower-degree moments are exact and only approximation are introduced in the highest-degree moment equation in the $P_N$ system. We therefore introduce the matrices for the ML closure as follows:
\begin{equation}\label{eq:ml-closure-matrix}
\begin{aligned}
\mathcal{A}^{\textrm{ML}} &=
\half
\begin{pmatrix}
0 & A_0^T \\
A_0 & 0 & A_1^T \\
& A_1 & 0 & A_2^T \\
& & \ddots & \ddots & \ddots \\
& & & A_{N-2} & 0 & A_{N-1}^T \\
G_0 & G_1 & \cdots & G_{N-2} & G_{N-1} & G_N
\end{pmatrix},
\\[0.5ex]
\mathcal{B}^{\textrm{ML}} &=
\half
\begin{pmatrix}
0 & B_0^T \\
B_0 & 0 & B_1^T \\
& B_1 & 0 & B_2^T \\
& & \ddots & \ddots & \ddots \\
& & & B_{N-2} & 0 & B_{N-1}^T \\
K_0 & K_1 & \cdots & K_{N-2} & K_{N-1} & K_N
\end{pmatrix},
\end{aligned}
\end{equation}
where the matrices $G_j,\ K_j \in \R^{(N+1)\times(j+1)}$ for $j=0,\dots,N$ are closure terms to be learned from data. Generally, they may depend nonlinearly on the moment vector $\vect{u}$, which will be approximated by neural networks. The corresponding ML moment closure system is
\begin{equation}\label{eq:closed-system}
    \partial_t \vect{u} + \mathcal{A}^{\textrm{ML}} \partial_x \vect{u} + \mathcal{B}^{\textrm{ML}} \partial_y \vect{u} = \mathcal{Q}\vect{u}.
\end{equation}

The $P_N$ system \eqref{eq:real-pn} is symmetric hyperbolic, with symmetrizer given by the identity matrix. To preserve the symmetrizable hyperbolicity for the new system \eqref{eq:closed-system}, we seek a block-diagonal symmetrizer
\begin{equation}
    \mathcal{S} = \diag(I_1,I_2,\dots,I_{N-1},H), \quad H\in \R^{(N+1)\times(N+1)}.
\end{equation}
Here $I_r$ denotes the $r\times r$ identity matrix and $H$ is a symmetric positive definite (SPD) matrix to be determined. The requirement is that $\mathcal{S}\mathcal{A}^{\textrm{ML}}$ and $\mathcal{S}\mathcal{B}^{\textrm{ML}}$ are both symmetric. To achieve this, we compute 
\begin{equation}
\mathcal{S}\mathcal{A}^{\textrm{ML}}
=
\half
\begin{pmatrix}
0 & A_0^T \\
A_0 & 0 & A_1^T \\
& A_1 & 0 & A_2^T \\
& & \ddots & \ddots & \ddots \\
& & & A_{N-2} & 0 & A_{N-1}^T \\
H G_0 & H G_1 & \cdots & H G_{N-2} & H G_{N-1} & H G_N
\end{pmatrix}.
\end{equation}
The upper left $N\times N$ block of $\mathcal{S}\mathcal{A}^{\textrm{ML}}$ is exactly the same as that of the original $P_N$ matrix $\mathcal{A}$, hence it is already symmetric. Therefore, to make $\mathcal{S}\mathcal{A}^{\textrm{ML}}$ symmetric, it is sufficient to compare the last block row and the last block
column, which leads to the following conditions for the closure terms. First, we have
\begin{equation}
    H G_j = 0, \quad 0\le j\le N-2.
\end{equation}
which implies that $G_j=0$ for $0\le j\le N-2$ due to the invertibility of $H$.
Second, we have
\begin{equation}
H G_{N-1} = A_{N-1},
\end{equation}
or equivalently,
\begin{equation}
G_{N-1}=H^{-1}A_{N-1}.
\end{equation}
And lastly, we have
\begin{equation}
H G_N = (H G_N)^T.
\end{equation}

Applying the same argument to $\mathcal{B}^{\textrm{ML}}$, we have
\begin{equation}
\mathcal{S}\mathcal{B}^{\textrm{ML}}
=
\half
\begin{pmatrix}
0 & B_0^T \\
B_0 & 0 & B_1^T \\
& B_1 & 0 & B_2^T \\
& & \ddots & \ddots & \ddots \\
& & & B_{N-2} & 0 & B_{N-1}^T \\
H K_0 & H K_1 & \cdots & H K_{N-2} & H K_{N-1} & H K_N
\end{pmatrix},
\end{equation}
and therefore
\begin{equation}
K_0=\cdots=K_{N-2}=0,
\qquad
K_{N-1}=H^{-1}B_{N-1},
\qquad
H K_N=(H K_N)^T.
\end{equation}

We summarize the characterization of the hyperbolicity constraints for the ML closure as follows:
\begin{thm}
    The ML moment closure model \eqref{eq:closed-system} with the coefficients
    given by \eqref{eq:ml-closure-matrix} is symmetrizable hyperbolic if and
    only if 
    \begin{equation}            
        G_j=0,\qquad K_j=0,\qquad 0\le j\le N-2,
    \end{equation}
    and any one of the following equivalent conditions holds:
    \begin{enumerate}
        \item There exists an SPD matrix $H$ such that
        \begin{equation}\label{eq:closure-constraints-1}
            G_{N-1}=H^{-1}A_{N-1},
            \qquad
            K_{N-1}=H^{-1}B_{N-1},
            \qquad
            HG_N=(HG_N)^T,
            \qquad
            HK_N=(HK_N)^T.
        \end{equation}

        \item There exist an SPD matrix $H$ and symmetric matrices $M_x$ and
        $M_y$ such that
        \begin{equation}\label{eq:closure-constraints-2}
            G_{N-1}=H^{-1}A_{N-1},
            \qquad
            K_{N-1}=H^{-1}B_{N-1},
            \qquad
            G_N = H^{-1}M_x,
            \qquad
            K_N = H^{-1}M_y.
        \end{equation}

        \item There exist an SPD matrix $H$ and symmetric matrices $M_x$ and
        $M_y$ such that
        \begin{equation}\label{eq:closure-constraints-3}
            G_{N-1}=H A_{N-1},
            \qquad
            K_{N-1}=H B_{N-1},
            \qquad
            G_N = H M_x,
            \qquad
            K_N = H M_y.
        \end{equation}
    \end{enumerate}
    In the implementation, we use the formulation
    \eqref{eq:closure-constraints-3}.
\end{thm}

\begin{proof}
The vanishing conditions for $G_j$ and $K_j$ with $0\le j\le N-2$, together with \eqref{eq:closure-constraints-1}, are exactly the relations obtained above by comparing the last block row and last block column of $\mathcal{S}\mathcal{A}^{\textrm{ML}}$ and $\mathcal{S}\mathcal{B}^{\textrm{ML}}$.

Condition \eqref{eq:closure-constraints-2} is just a rewriting of \eqref{eq:closure-constraints-1}, if we set
\begin{equation}    
    M_x:=HG_N, \qquad M_y:=HK_N,
\end{equation}
then $M_x$ and $M_y$ are symmetric when $HG_N$ and $HK_N$ are symmetric, and conversely \eqref{eq:closure-constraints-1} follows from
\eqref{eq:closure-constraints-2} by substituting $HG_N=M_x$ and $HK_N=M_y$.

Finally, \eqref{eq:closure-constraints-2} and \eqref{eq:closure-constraints-3}
are equivalent after the change of variable $\widetilde H=H^{-1}$ and the
relabeling $\widetilde H\mapsto H$, since the inverse of an SPD matrix is also
SPD.
\end{proof}

To guarantee the constraints of hyperbolicity in \eqref{eq:closure-constraints-3}, we may set
\begin{equation}
H = LL^T + \varepsilon I,
\end{equation}
where $L\in \R^{(N+1)\times(N+1)}$ is a lower triangular matrix and $\varepsilon>0$ is a small constant to ensure the positive definiteness of $H$. Then $H$ is always an SPD matrix for any $L$. To define $M_x$ and $M_y$, we take the symmetric part of arbitrary matrices $L_x$ and $L_y$:
\begin{equation}
    M_x = \half(L_x + L_x^T),
    \qquad
    M_y = \half(L_y + L_y^T),
\end{equation}
Note that $L$, $L_x$, and $L_y$ will be taken as nonlinear matrix-valued functions of the moment vector $\vect{u}$:
\begin{equation}
L = L(\vect{u}),\qquad
L_x = L_x(\vect{u}),\qquad
L_y = L_y(\vect{u}),
\end{equation}
which will be parameterized by multilayer perceptrons (MLPs). With this parametrization, the hyperbolicity constraints are automatically satisfied for any choice of the parameters in $L$, $L_x$, and $L_y$. Therefore, the closure terms $G_{N-1}$, $K_{N-1}$, $G_N$, and $K_N$ can be learned from data without worrying about the loss of hyperbolicity. The neural network architecture is illustrated in Fig.~\ref{fig:nn-architecture}.

\begin{figure}[ht]
\centering
{
\includegraphics[width=0.7\textwidth,trim={238pt 226pt 257pt 134pt},clip]{\figureprefix 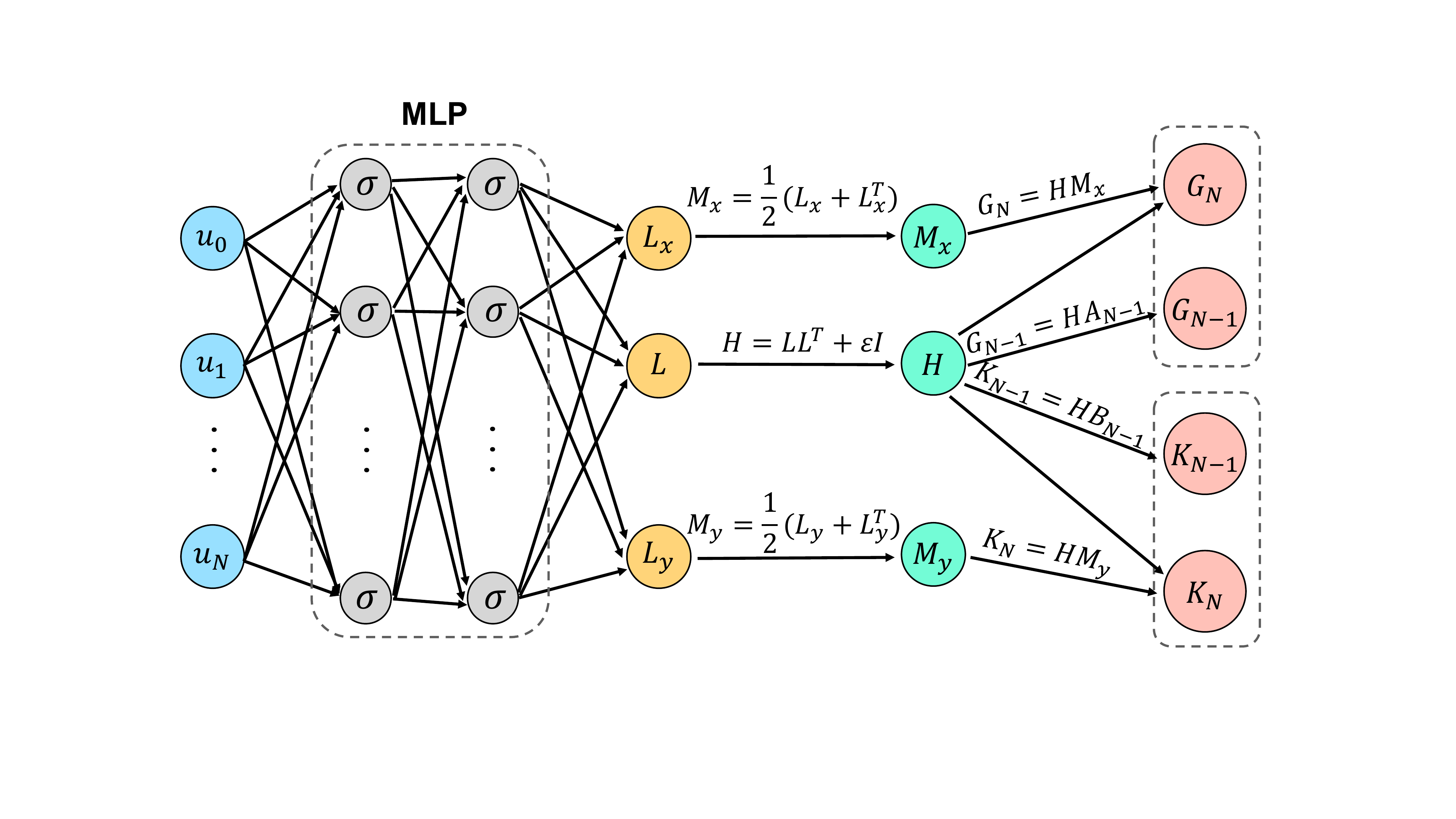}
}
\caption{Architecture of the hyperbolicity-preserving neural network closure model. The input is the moment vector $\vect{u} = (\vect{u}_0,\dots,\vect{u}_N)^T$, followed by the MLPs to compute $L\in\mathbb{R}^{(N+1)\times(N+1)}$, $L_x\in\mathbb{R}^{(N+1)\times(N+1)}$, and $L_y\in\mathbb{R}^{(N+1)\times(N+1)}$. Then $M_x$ and $M_y$ are obtained by symmetrizing $L_x$ and $L_y$, and $H$ is computed as $LL^T+\varepsilon I$. Finally, the closure terms $G_{N-1}$ and $G_N$ in the $x$-direction are computed by multiplying $H$ with $A_{N-1}$ and $M_x$, respectively, and the closure terms $K_{N-1}$ and $K_N$ in the $y$-direction are computed by multiplying $H$ with $B_{N-1}$ and $M_y$, respectively. With this architecture, the hyperbolicity constraints in \eqref{eq:closure-constraints-3} are automatically satisfied for any choice of the parameters in the MLPs.
}
\label{fig:nn-architecture}
\end{figure}

To derive the loss function, we consider the difference between the exact evolution and the closed evolution. We can write down the exact evolution of the highest-order moment $\vect{u}_N$ (without closure) as:
\begin{equation}
\partial_t \vect{u}_N 
+ \half A_{N-1} \partial_x \vect{u}_{N-1} 
+ \half A_{N}^T \partial_x \vect{u}_{N+1}
+ \half B_{N-1} \partial_y \vect{u}_{N-1} 
+ \half B_{N}^T \partial_y \vect{u}_{N+1} = Q_N \vect{u}_N.
\end{equation} 
We replace this with the closed evolution
\begin{equation}
\partial_t \vect{u}_N 
+ \half G_{N-1} \partial_x \vect{u}_{N-1}
+ \half G_{N} \partial_x \vect{u}_{N}
+ \half K_{N-1} \partial_y \vect{u}_{N-1} 
+ \half K_{N} \partial_y \vect{u}_{N} = Q_N \vect{u}_N.
\end{equation}
Compare the two equations, we can define the loss function as
\begin{equation}\label{eq:loss-function}
\begin{aligned}
\mathcal{L}
= {}&
\Bigl\|
(G_{N-1}-A_{N-1}) \partial_x \vect{u}_{N-1}
+ G_{N} \partial_x \vect{u}_{N}
- A_{N}^T \partial_x \vect{u}_{N+1}
\\
&\qquad
+ (K_{N-1}-B_{N-1}) \partial_y \vect{u}_{N-1}
+ K_{N} \partial_y \vect{u}_{N}
- B_{N}^T \partial_y \vect{u}_{N+1}
\Bigr\|^2.
\end{aligned}
\end{equation}

\section{Numerical experiments}\label{sec:numerical-experiments}

In this section, we conduct numerical experiments to investigate the performance of our ML moment closure. 
We first illustrate the training and evaluation procedure. We then begin with a simple single-mode sine-wave experiment that serves as a calibration test for the training pipeline and for the choice of retained order. We then move to a family of random multi-mode sine-wave initial conditions, which is the main in-family generalization test. Finally, we consider training on multi-mode data with varying material parameters in order to test robustness with respect to both the initial condition and the medium.

\subsection{Training and Evaluation Procedure}

In this part, we describe the training and evaluation procedure for the ML closure model. First, to generate the training data, we make use of a $P_N$ solver based on StaRMAP \cite{seibold2014starmap}, a second-order staggered-grid method implemented in MATLAB. Here, we run the solver at a sufficiently large moment order to obtain a good approximation to the kinetic solution. For each trajectory, we save the moments $\vect{u}_l$ for $0\le l\le {N+1}$ at the specified time steps.
Then, we preprocess the training data to compute the spatial derivatives $\partial_x \vect{u}_l$ and $\partial_y \vect{u}_l$ for $l = N-1, N, N+1$ using a second-order finite difference method.

Then, we train the ML closure model by minimizing the residual-based loss in \eqref{eq:loss-function} on the processed reference trajectories. The input to the network is the moment vector $(\vect{u}_0,\dots,\vect{u}_N)$, while the derivatives of $\vect{u}_{N-1}$, $\vect{u}_N$, and $\vect{u}_{N+1}$ are used inside the loss to evaluate the mismatch in the last retained block.

For the hyperbolic model, the network consists of one branch generating an SPD matrix $H(\vect{u})$ and two branches generating symmetric matrices $M_x(\vect{u})$ and $M_y(\vect{u})$. These outputs determine the learnable closure blocks through the structural formulas in \eqref{eq:closure-constraints-3}, so the architecture enforces the symmetry and positivity constraints by construction.

In the numerical experiments, we use the mini-batch AdamW optimizer with a learning rate of $10^{-3}$ and a batch size of 1024. The MLP architecture consists of 2 hidden layers with 64 neurons in each layer, if not specified otherwise. We train for 1000 epochs and select the checkpoint with the smallest validation loss for evaluation.

\subsection{Task 1: a single-mode sine wave solution}

In this task, the setup is relatively simple. Its purpose is not to demonstrate the strongest possible performance of the ML closure model, but rather to calibrate the workflow on a clean low-complexity example. Here, we generate the sine wave data in which only the zeroth-order moment is nonzero initially and all higher moments are zero. To keep the specific intensity positive, we take the initial condition of the zeroth moment to be
\begin{equation}\label{eq:single-mode-sine-initial-condition}
    u_0(x,y,0) = \sin\bigl(\pi(x + y)\bigr) + 2,
\end{equation}
The computational domain is $[-1, 1]^2$ with periodic boundary conditions, and the material parameters are $\sigma_a=0$ and $\sigma_{s}=1$.

The first step in this task is to determine a reference moment order and generate the training data with the initial condition \eqref{eq:single-mode-sine-initial-condition}. We first run the StaRMAP solver for different moment orders and compare the solutions in Fig.~\ref{fig:reference-sine-kx1-ky1}. We observe that $P_2$ has a large error in the zeroth moment at $t=1$, while $P_3$ already shows a significant improvement. The higher-order models $P_4$, $P_5$, and $P_6$ are nearly indistinguishable from the $P_{10}$ model. This comparison suggests that $P_{10}$ is a good reference solution for this case, and that the lower-order models are already close to the reference in the zeroth moment when $N\ge 4$.

\begin{figure}[ht]
\centering
{
\includegraphics[width=0.5\textwidth]{\figureprefix 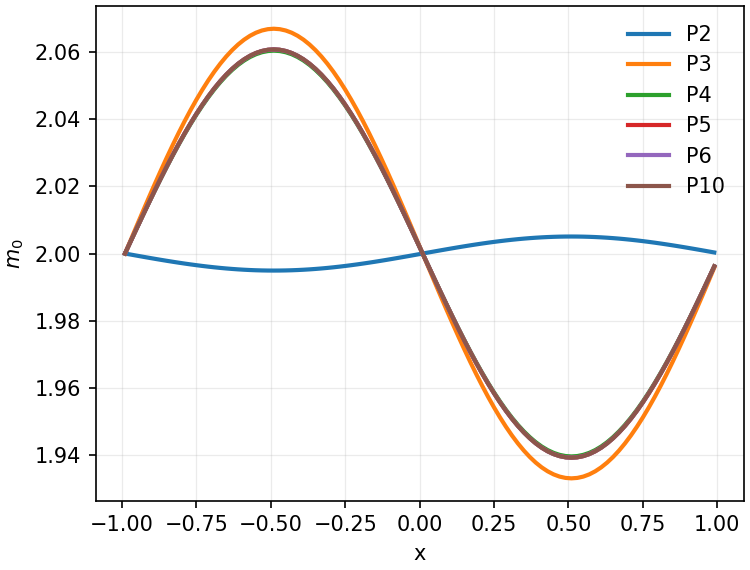}
}
\caption{Task 1: Comparison of $P_N$ models for the single-mode sine wave initial condition in \eqref{eq:single-mode-sine-initial-condition}. The figure shows the 1D cuts at $y=0$ and $t=1$ for the zeroth-order moment $u_0$ from $P_2,P_3,P_4,P_5,P_6$, and $P_{10}$. We observe that $P_2$ exhibits a large error, and $P_3$ already shows a significant improvement. The other higher-order models $P_4$, $P_5$, and $P_6$ are all nearly indistinguishable from the $P_{10}$ model.}
\label{fig:reference-sine-kx1-ky1}
\end{figure}

The second step is to train the closure model using the generated training data from the $P_{10}$ reference solution at $t=0, 0.1, 0.2, \dots, 1$. After choosing a retained order $N$ in the ML model, we use the reference solution as the training data. At each grid point and each recorded time, one sample contains the retained state $(\vect{u}_0,\dots,\vect{u}_N)$ together with the derivative blocks required by the residual formulation, namely $\partial_x \vect{u}_{N-1}$, $\partial_y \vect{u}_{N-1}$, $\partial_x \vect{u}_N$, $\partial_y \vect{u}_N$, $\partial_x \vect{u}_{N+1}$, and $\partial_y \vect{u}_{N+1}$. The neural network takes $(\vect{u}_0,\dots,\vect{u}_N)$ as input and predicts the matrices $H(\vect{u})$, $M_x(\vect{u})$, $M_y(\vect{u})$, from which the last-row closure blocks are assembled by $G_{N-1}=H A_{N-1}$, $K_{N-1}=H B_{N-1}$, $G_N = H M_x$, $K_N = H M_y$. We then minimize the residual loss on these pointwise samples. We train for 1000 epochs and select the checkpoint with the smallest validation loss for evaluation.

The third step is to test the trained closure. We insert the ML last-row correction into the 2D $P_N$ solver and run a full prediction with the same sine wave initial condition \eqref{eq:single-mode-sine-initial-condition}. We then compare the results from the ML moment closure against both the linear $P_N$ baseline and the deterministic $P_{10}$ reference solution.

The results with $N=2$ at $t=1$ are shown in Fig.~\ref{fig:sine-kx1-ky1-rollout-compare}(a). The ML moment closure with $N=2$ greatly improves over the linear $P_2$ model, with only a small deviation near the extremes of the sine wave. The relative $L^2$ error of the zeroth order moment with respect to $P_{10}$ is $1.75\times10^{-4}$ for the ML closure, while it is $2.33\times 10^{-2}$ for the linear $P_2$ model. Thus, it improves the error by roughly a factor of $133$. This indicates that, in this simple case, learning a closure greatly outperforms the linear $P_N$ closure.
      
The results with $N=3$ at $t=1$ are shown in Fig.~\ref{fig:sine-kx1-ky1-rollout-compare}(b). In this case, the linear $P_3$ model already performs well, with only a small deviation from the $P_{10}$ reference near the extremes. The ML closure improves over the linear $P_3$ model, with perfect agreement with the reference solution. The relative $L^2$ error of the zeroth moment with respect to $P_{10}$ is $5.04\times 10^{-5}$ for the ML closure, while it is $2.17\times 10^{-3}$ for the linear $P_3$ model. Thus, the ML closure reduces the error by roughly a factor of $43$ relative to the linear $P_3$ model. In this case, learning the closure still provides an improvement, but the baseline is already quite accurate, so the relative gain is smaller than the case of $N=2$.

\begin{figure}[htbp]
    \centering
    \begin{subfigure}[b]{0.48\textwidth}
        \centering
        \includegraphics[width=\textwidth]{\figureprefix 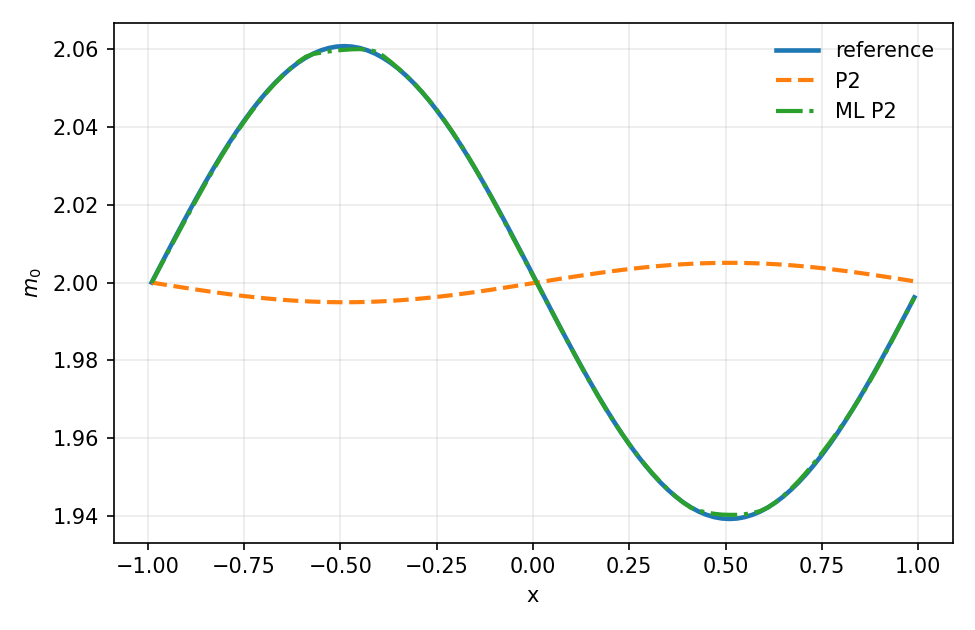}
        \caption{$N=2$}
    \end{subfigure}
    \hfill
    \begin{subfigure}[b]{0.48\textwidth}
        \centering
        \includegraphics[width=\textwidth]{\figureprefix 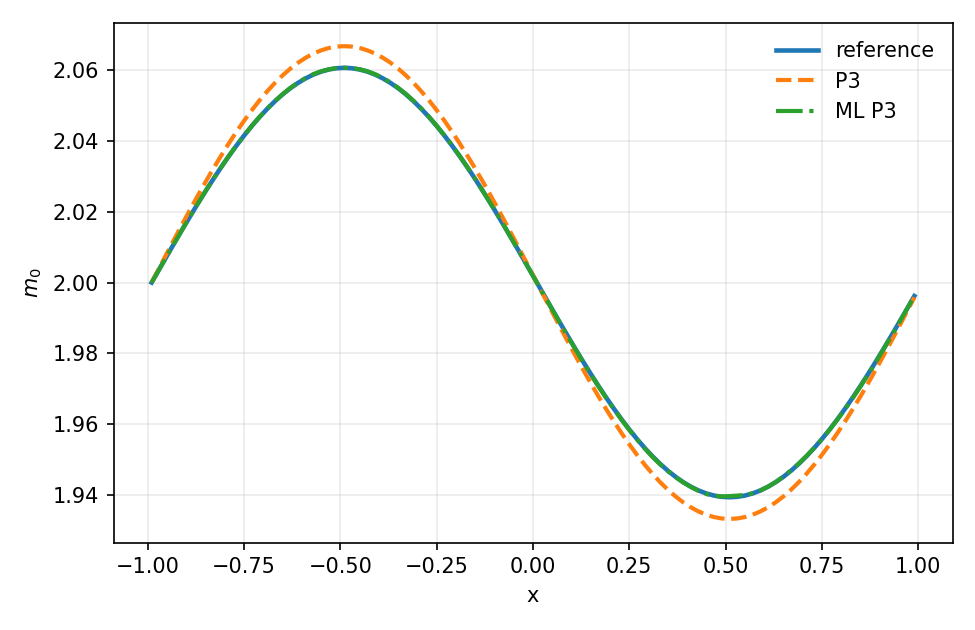}
        \caption{$N=3$}
    \end{subfigure}
    \caption{Task 1: Comparison for the ML closures and the linear $P_N$ models. The left panel compares the linear $P_2$ model, the ML $P_2$ model, and the $P_{10}$ reference at $t=1$ along the line $y=0$. The right panel shows the analogous comparison for the case of $N=3$. In $N=2$, the ML closure improves a lot over the linear $P_2$ model, with only a minor mismatch near the extremes. In $N=3$, the ML closure is better than the linear $P_3$ model.}
    \label{fig:sine-kx1-ky1-rollout-compare}
\end{figure}

In summary, Task~1 provides a controlled validation of the proposed ML closure framework on a simple single-mode sine wave problem. Using the $P_{10}$ data for training, the ML closures for both $N=2$ and $N=3$ yield more accurate predictions than their corresponding linear $P_N$ models. For the zeroth-order moment at $t=1$, the relative $L^2$ error is reduced by approximately two orders of magnitude for $N=2$ and by more than one order of magnitude for $N=3$. These results demonstrate that, even in this low complexity setting, the ML closure is able to recover the effect of unresolved higher-order moments and substantially improve the accuracy relative to the standard $P_N$ closure.

\subsection{Task 2: a family of multi-mode sine wave solutions with fixed material parameters}

This numerical experiment broadens the training distribution from one fixed wave in Task 1 to a family of random multi-mode sine wave initial conditions. The purpose is to answer the question of whether an ML closure can improve the predictions across a diverse family of solutions. We use the initial condition in the form of a truncated Fourier series with random coefficients and phase shifts \cite{han2019uniformly}:
\begin{equation}\label{eq:multi-sine-ic}
    u_0(x,y,0) = \sum_{m, n=1}^{k_{\max}} a_{m, n} \sin\bigl(\pi(m x + n y) + \phi_{m, n}\bigr) + a_0.
\end{equation}
Here $a_{m,n}$ are random variables sampled uniformly from $[-\frac{1}{mn}, \frac{1}{mn}]$, $0\le \phi_{m,n}\le 2\pi$ is a random phase shift, and $a_0 = \sum_{m, n=1}^{k_{\max}} \frac{1}{mn} + c$ with $c$ sampled uniformly from $[0,1]$ to ensure the initial density is positive. In the current experiment, we take $k_{\max}=10$ to include a large number of modes in the initial condition, which leads to a more challenging test in both training and prediction. The computational domain is $[-1, 1]^2$ with periodic boundary conditions, and the material parameters are fixed as $\sigma_a=0$ and $\sigma_{s}=1$.

Since the initial condition \eqref{eq:multi-sine-ic} is more complex than the single-mode sine wave, the reference moment system needs to be solved at a higher order to ensure the accuracy of the training data. Numerically, we observe that the $P_{50}$ model is already very close to the $P_{100}$ reference in the zeroth-order moment for this family of initial conditions, and the error does not decrease significantly when the moment order increases beyond $50$. Thus, the $P_{50}$ model can generate good reference solutions for this experiment, and we use it to generate the training data for the ML closure.

We first run the StaRMAP solver for the $P_{50}$ model on the time interval $[0,1]$ with snapshots at $t=0,0.1,\dots,1$. To visualize the initial condition of the generated multi-sine data, we plot the initial conditions of the zeroth-order moments for five different random seeds in Fig.~\ref{fig:multi-sine-seed-initial-condition}. We observe that the initial conditions are quite diverse across different seeds, which indicates that the training data covers a wide range of initial conditions.

\begin{figure}[ht]
\centering
{%
\includegraphics[width=0.5\textwidth]{\figureprefix 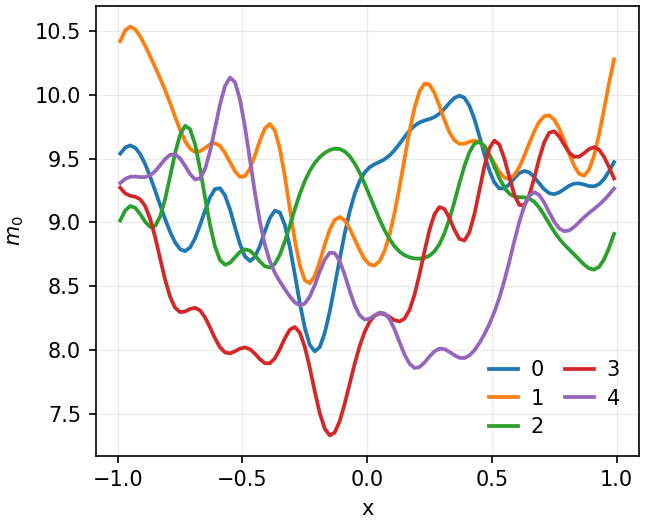}%
}
\caption{Task 2: Visualization of the initial conditions for the generated multi-sine $P_{50}$ data for $k_{\max}=10$. It shows the 1D cuts at $y=0$ of the zeroth-order moment at the initial time $t=0$ for five realizations with random seeds $0,1,2,3,4$. This illustrates the diversity of the random multi-sine initial conditions.}
\label{fig:multi-sine-seed-initial-condition}
\end{figure}

The second step is to train the ML closure model using the generated training data. To find the optimal hyperparameters for the ML closure, we perform a grid search over the hidden width and the number of hidden layers. In particular, we take the number of hidden layers $l \in \{1,2,3,4,5\}$ and the number of neurons per layer $w \in \{16,32,64,128,256\}$. For each fixed retained order, Fig.~\ref{fig:multisine-hparam-tuning} plots the best validation relative loss versus the number of layers, with one curve for each hidden width. Across all four retained orders $N=3,4,5,9$, increasing the width and depth consistently improves the validation error over the tested range. In particular, within this sweep, the best configuration is obtained at width $w = 256$ and depth $l = 5$, with best validation relative losses approximately $3.82\times 10^{-2}$, $1.12\times 10^{-2}$, $7.40\times 10^{-3}$, and $6.61\times 10^{-3}$, respectively. Due to the limitation of computational resources, we do not explore larger widths and depths in this experiment, but the observed trend suggests that further improvement may be possible by increasing the model capacity. In the numerical experiments in the rest of this section, we use the best hyperparameters  found in this sweep with width $256$ and depth $5$ for all retained orders.
\begin{figure}[htbp]
\centering
\begin{subfigure}[t]{0.48\textwidth}
    \centering
    \includegraphics[width=\textwidth]{\figureprefix 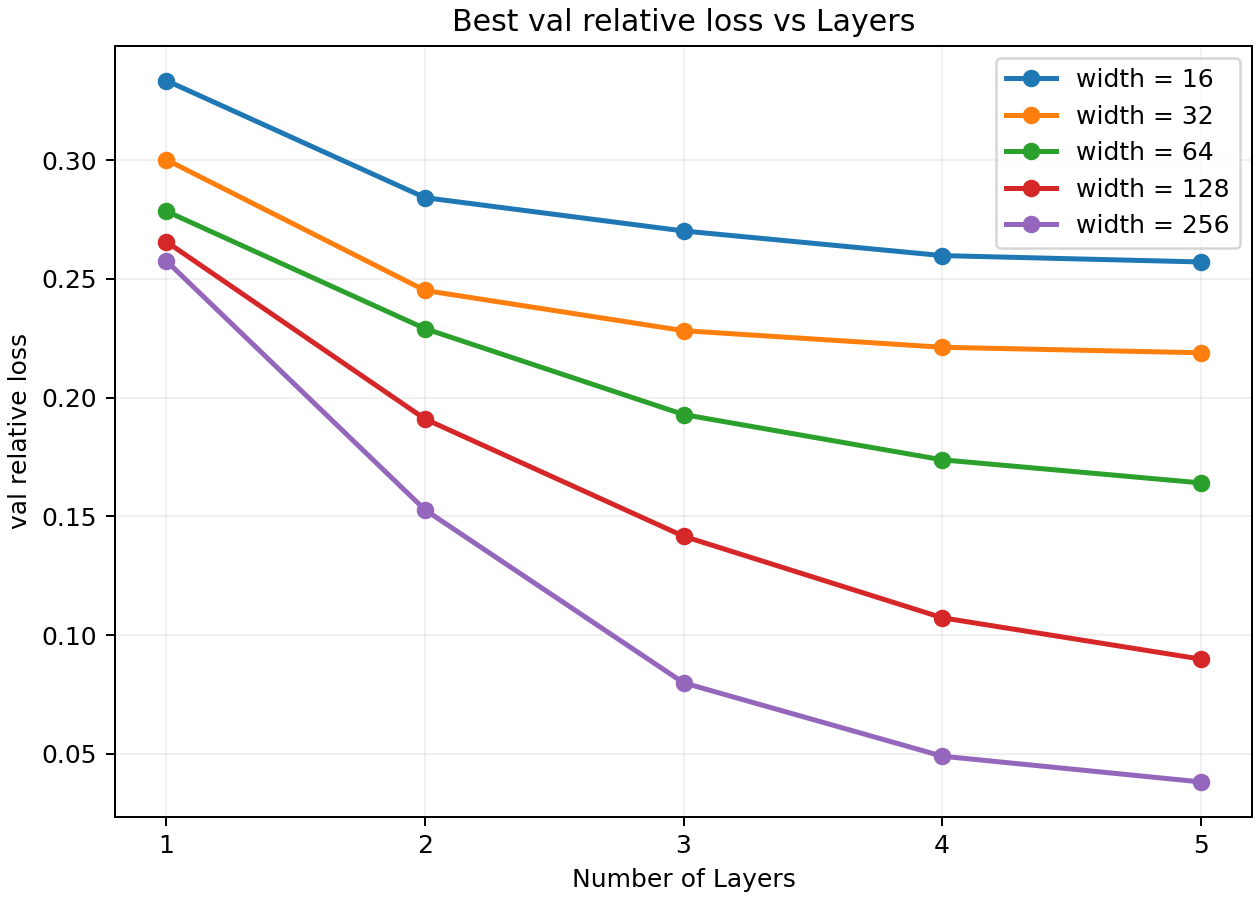}
    \caption{$N=3$}
\end{subfigure}
\hfill
\begin{subfigure}[t]{0.48\textwidth}
    \centering
    \includegraphics[width=\textwidth]{\figureprefix 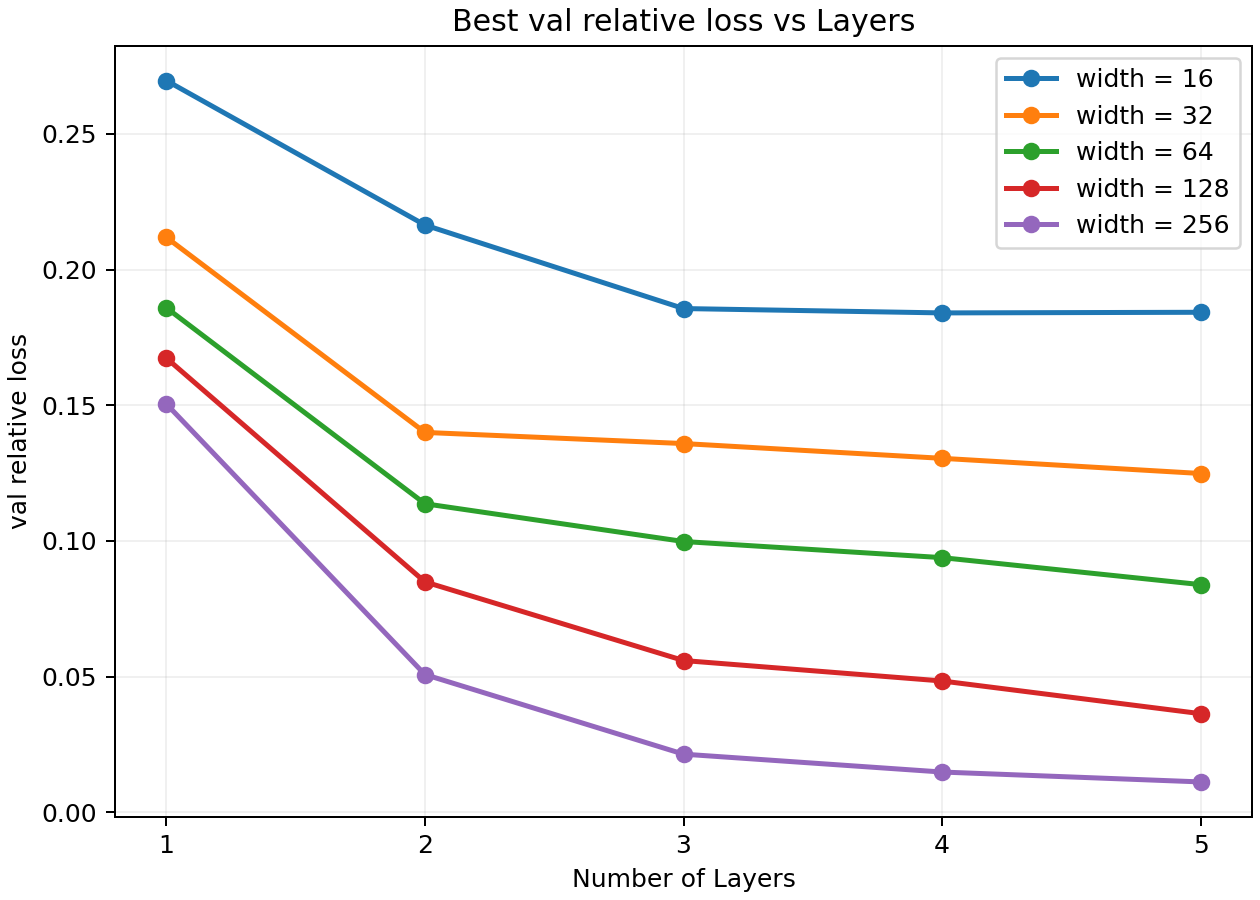}
    \caption{$N=5$}
\end{subfigure}

\vspace{0.5em}

\begin{subfigure}[t]{0.48\textwidth}
    \centering
    \includegraphics[width=\textwidth]{\figureprefix 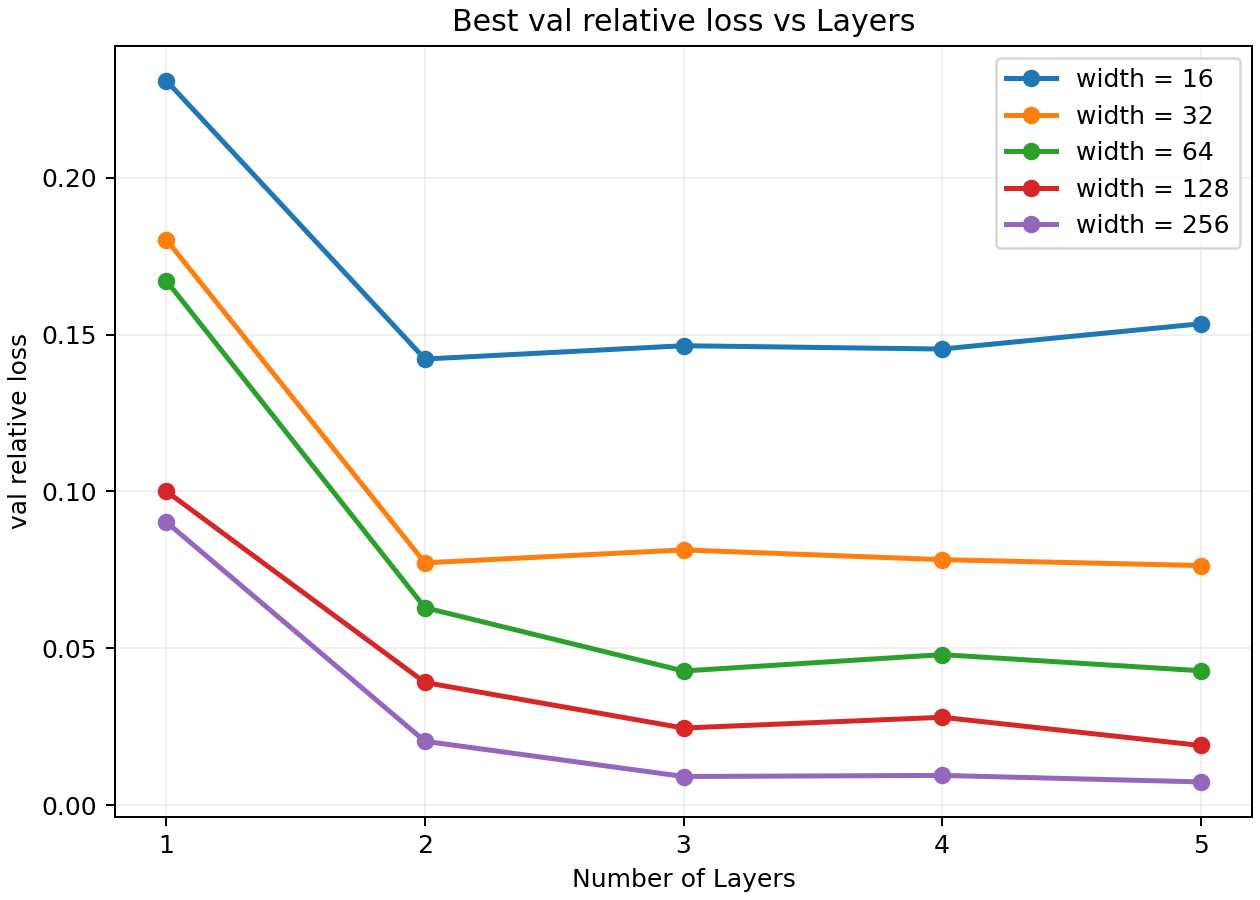}
    \caption{$N=7$}
\end{subfigure}
\hfill
\begin{subfigure}[t]{0.48\textwidth}
    \centering
    \includegraphics[width=\textwidth]{\figureprefix 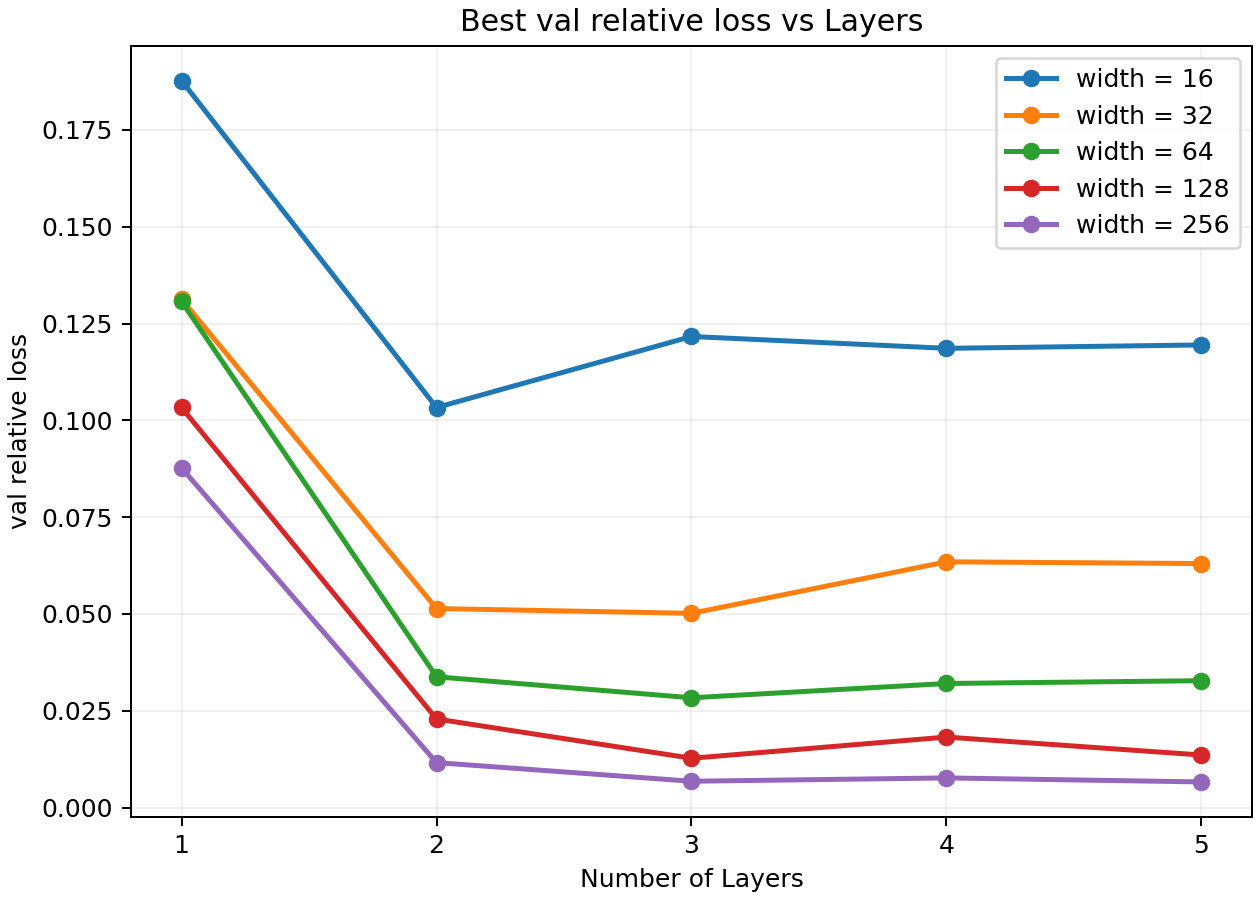}
    \caption{$N=9$}
\end{subfigure}
\caption{Task 2: Hyperparameter sweep for the multi-sine training problem with fixed material. Each panel corresponds to one retained moment order $N=3,5,7,9$, and shows the best validation relative loss as a function of the number of hidden layers for widths $w=16,32,64,128,256$. Over the tested range, larger widths and deeper networks generally lead to lower validation error, with the strongest performance in this sweep achieved by the $w=256$, $L=5$ models.}
\label{fig:multisine-hparam-tuning}
\end{figure}

The third step is to test the trained closure by running the prediction from the same family of multi-sine initial condition and comparing the ML closure against both the linear $P_N$ baseline and the $P_{50}$ reference solution. 
We first test the performance of the ML closure on the same random seed used in training. In particular, we compare the ML closure with $N=3$ against the linear $P_3$ baseline on the trajectory with seed $0$. In Fig.~\ref{fig:multisine-p3-testing-contour}, we observe that, the $P_3$ model shows large oscillations, while the ML closure produces a much smoother solution that is closer to the $P_{50}$ reference. We also notice that the ML closure has some visible error in approximating the reference solution. This is probably because the order $N=3$ is still relatively low for this complex multi-sine problem, which makes it more difficult for the ML closure to learn the effect of the unresolved higher-order moments. 
\begin{figure}[htbp]
    \centering
    \begin{subfigure}[b]{0.67\textwidth}
        \centering
        \includegraphics[width=\textwidth]{\figureprefix 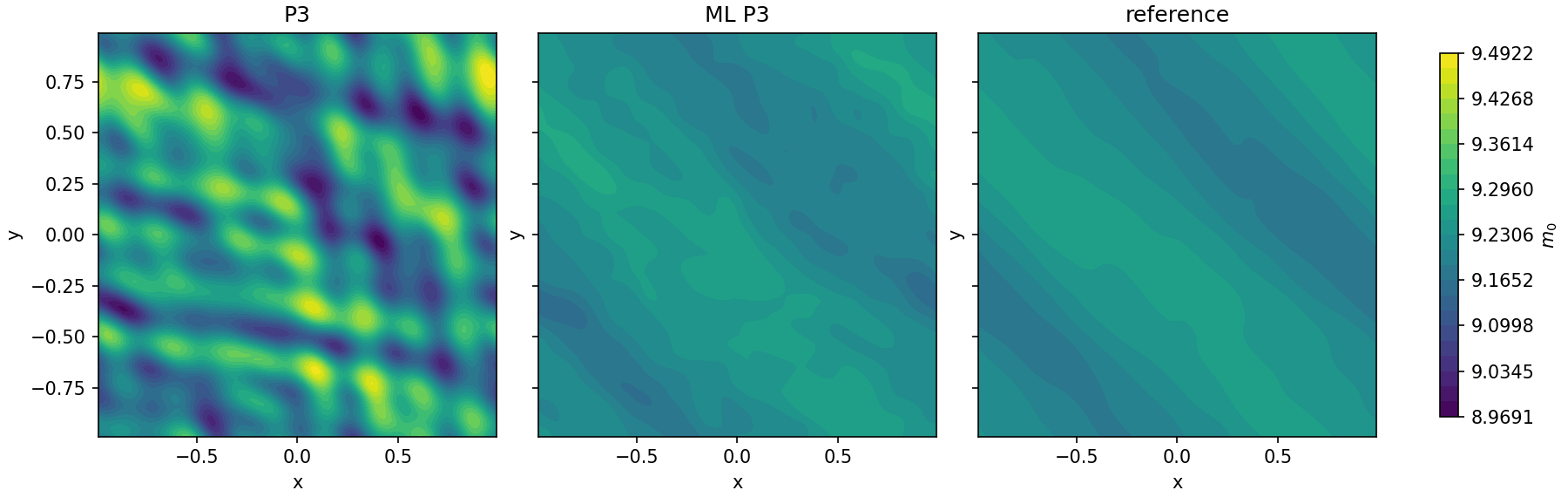}
        \caption{2D contour map}
    \end{subfigure}
    \hfill
    \begin{subfigure}[b]{0.32\textwidth}
        \centering
        \includegraphics[width=\textwidth]{\figureprefix 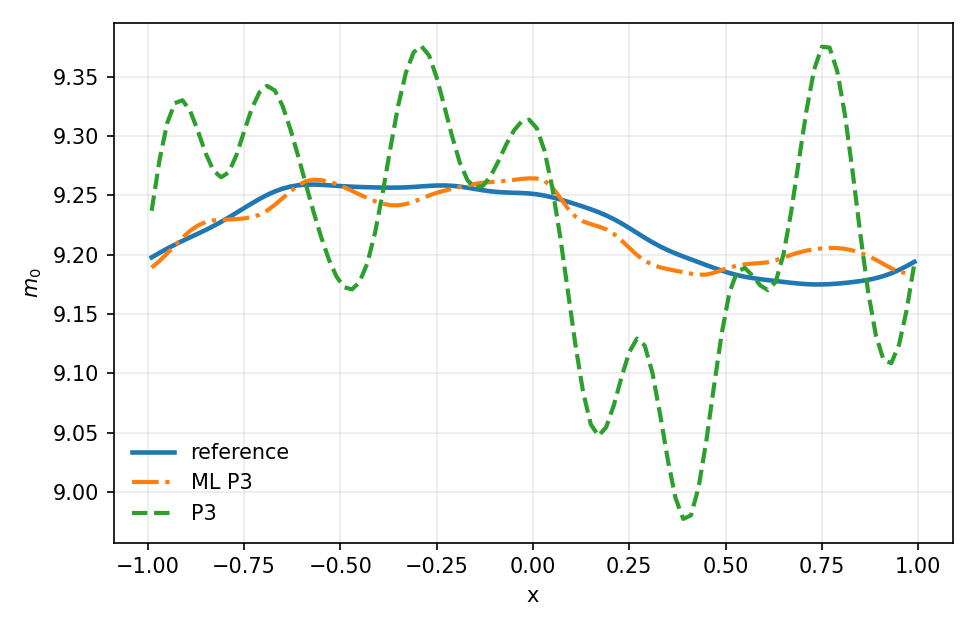}
        \caption{1D cut at $y=0$}
    \end{subfigure}
    \caption{Task 2: Comparison for the ML closure with $N=3$ and $P_3$ closure at $t=1$ and random seed $0$. The left figure shows the 2D contour map of the $P_3$ closure, the ML closure with $N=3$, and the $P_{50}$ reference. The right figure shows the 1D cut at $y=0$ for the same three solutions. We observe that the ML solution is much closer to the $P_{50}$ reference than the linear baseline in the full spatial field.}
    \label{fig:multisine-p3-testing-contour}
\end{figure}

Then, we increase the retained order to $N=5$ and test the ML closure against the linear $P_5$ baseline on the same trajectory with random seed $0$. The results are shown in Fig.~\ref{fig:multisine-p5-testing-contour}. We observe that, the linear $P_5$ model still shows some oscillations, while the ML closure produces a much smoother solution that is closer to the $P_{50}$ reference. The ML closure also has some visible error in approximating the reference solution, but it is better than the ML closure with $N=3$ and substantially better than the linear $P_5$ model. This suggests that, as the retained order increases, learning the closure becomes more effective in improving the accuracy.
\begin{figure}[htbp]
    \centering
    \begin{subfigure}[b]{0.67\textwidth}
        \centering
        \includegraphics[width=\textwidth]{\figureprefix 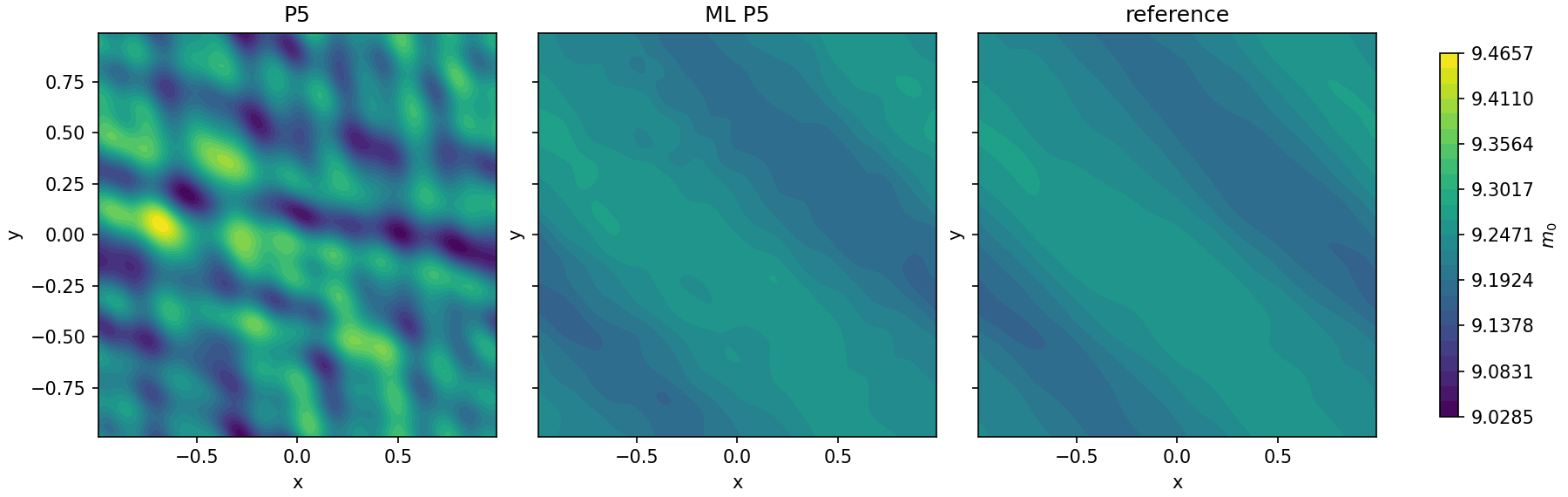}
        \caption{2D contour map}
    \end{subfigure}
    \hfill
    \begin{subfigure}[b]{0.32\textwidth}
        \centering
        \includegraphics[width=\textwidth]{\figureprefix 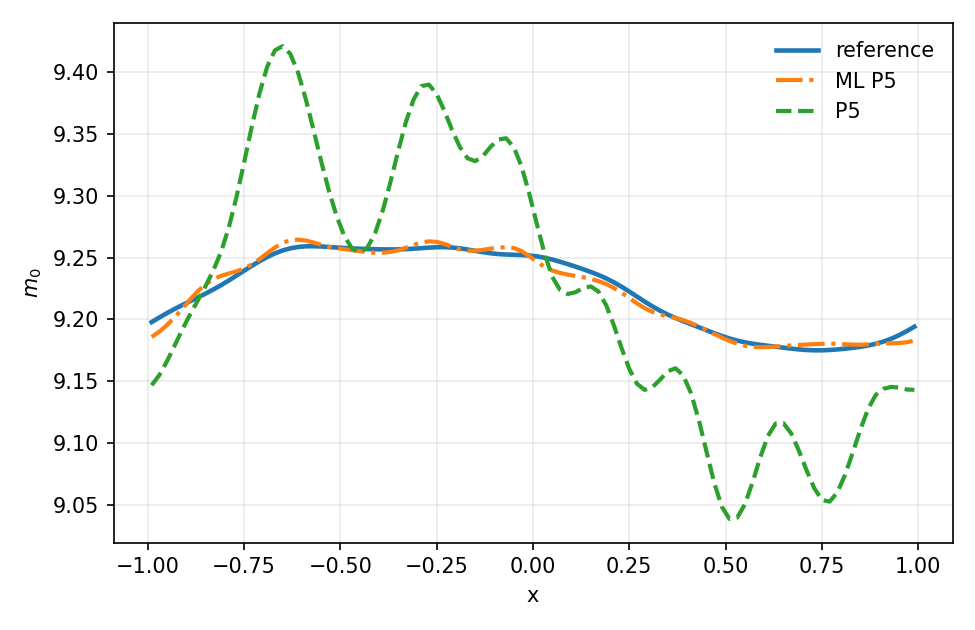}
        \caption{1D cut at $y=0$}
    \end{subfigure}
    \caption{Task 2: Comparison for the ML closure with $N=5$ and $P_5$ closure at $t=1$ and random seed $0$. The left figure shows the linear $P_5$ closure, the ML closure with $N=5$, and the $P_{50}$ reference. The right figure shows the 1D cut at $y=0$ for the same three solutions. 
    We observe that the ML solution is much closer to the $P_{50}$ reference than the linear baseline in the full spatial field.}
    \label{fig:multisine-p5-testing-contour}
\end{figure}

To test the generalization of the ML closure, we also run other realizations with different random seeds in the same family of multi-sine initial conditions but not used in training. In particular, we test the ML closure with $N=5$ on the trajectory with random seed $11$, which is outside the training set, and compare it against the linear $P_5$ baseline and the $P_{50}$ reference. The results are shown in Fig.~\ref{fig:multisine-p5-testing-contour-seed11}. We observe that, even on this unseen trajectory, the ML closure still produces a much smoother solution that is closer to the $P_{50}$ reference than the linear $P_5$ model. This suggests that the ML closure can generalize to new initial conditions within the same family of multi-sine solutions, and it can still improve over the linear baseline in this setting.
\begin{figure}[htbp]
    \centering
    \begin{subfigure}[b]{0.67\textwidth}
        \centering
        \includegraphics[width=\textwidth]{\figureprefix 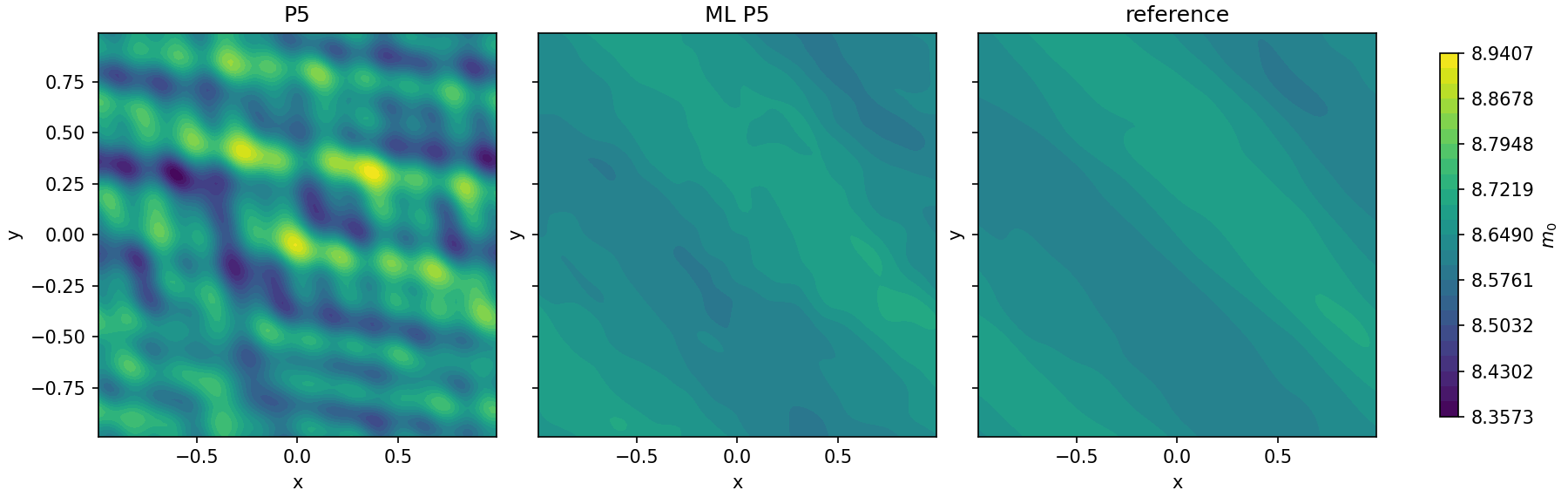}
        \caption{2D contour map}
    \end{subfigure}
    \hfill
    \begin{subfigure}[b]{0.32\textwidth}
        \centering
        \includegraphics[width=\textwidth]{\figureprefix 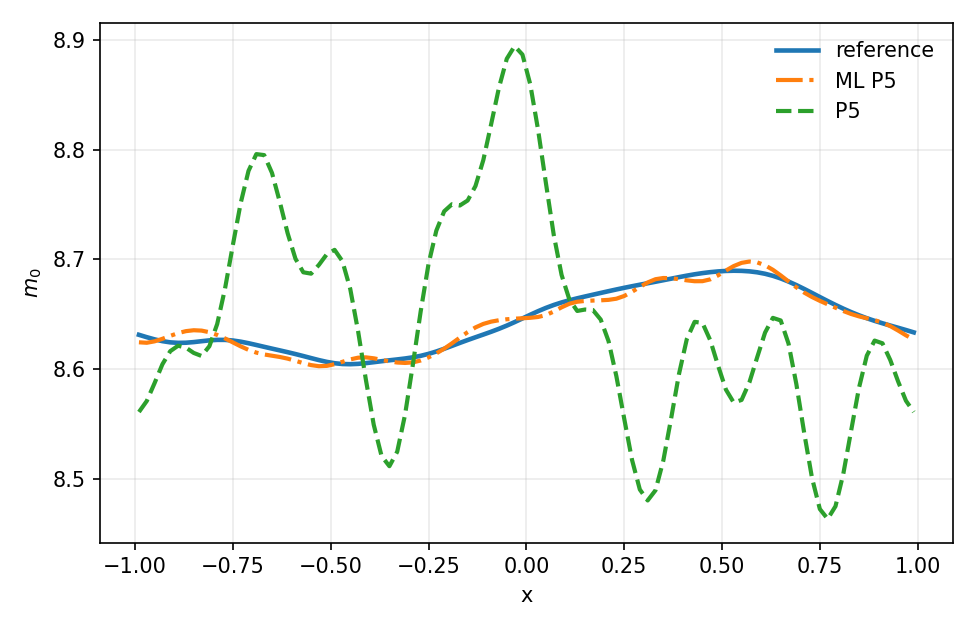}
        \caption{1D cut at $y=0$}
    \end{subfigure}
    \caption{Task 2: Comparison for the ML closure with $N=5$ and $P_5$ closure at $t=1$ and random seed $11$. The left figure shows the ML closure with $N=5$, the linear $P_5$ closure baseline, and the $P_{50}$ reference. The right figure shows the 1D cut at $y=0$. We observe that the ML solution is closer to the $P_{50}$ reference than the linear baseline in the full spatial field.}
    \label{fig:multisine-p5-testing-contour-seed11}
\end{figure}

In summary, these results show that, for the family of multi-mode sine wave solutions with fixed material coefficients, the ML closure consistently improves over the linear $P_N$ model. The test on the unseen trajectory further shows that the improvement is not restricted to a single training realization, but extends to new members of the same multi-sine family. Taken together, these experiments provide evidence that the ML closure captures useful unresolved moment information and shows a clear generalization improvement over the standard linear $P_N$ closure.

\subsection{Task 3: a family of multi-mode sine wave solutions with different material parameters}

The previous two tasks keep the material coefficients fixed and only vary the initial condition. This setting is useful for testing whether the ML closure can improve on the linear $P_N$ baseline within one homogeneous medium, but it does not yet address a more important practical question: whether the ML closure can remain accurate across different scattering and absorption regimes. The purpose of this third task is therefore to enlarge the training distribution in two directions at once: (1) the initial data are still taken from the random multi-sine family \eqref{eq:multi-sine-ic}, so that the spatial profiles remain diverse; (2) the homogeneous material parameters are now also varied from one training trajectory to another, by sampling both $\sigma_a$ and $\sigma_{s}$.

To generate the training data for this task, we follow the same procedure as in Task~2 to run the StaRMAP solver for the $P_{50}$ model on the time interval $[0,1]$ with snapshots at $t=0,0.1,\dots,1$. However, instead of fixing the material parameters in Task~2, we now sample them from prescribed ranges for each random multi-sine realization. In particular, we take $\sigma_a \in [0, 1]$ and $\sigma_{s} \in [1, 10]$ to cover a range of absorption and scattering regimes. We take the random seeds $0,1,\dots,99$ to generate a total of $100$ trajectories with different multi-sine initial conditions and different material parameters.

In practice, storing all 100 seeds with each seed containing 11 snapshots of the full $P_{50}$ moments and their derivatives can exceed the available local storage, especially when the retained order $N$ is large. 
We therefore replace the one-shot data generation pipeline by a streaming selection procedure: after generating one seed, we compute the score of the corresponding snapshot file based on the magnitude and spatial variation of the omitted degree block $\vect{u}_{N+1}$, which is the leading term in the $P_N$ closure residual:
\begin{equation}
    s = \|\vect{u}_{N+1}\| + \|\nabla \vect{u}_{N+1}\|.
\end{equation}
Hence we keep the files whose next order moment is both large in magnitude and rapidly varying in space, since these are precisely the regions where the closure model is expected to matter most. Maintaining a storage budget of $K$ files is then straightforward: after each newly processed seed, we recompute the scores of the currently available derivative files, keep the top $K$ files, and delete the rest. This yields an online top-$K$ selection rule that never requires the full $0,\dots,99$ dataset to be stored at once.

Fig.~\ref{fig:multisine-hard-file-selection} illustrates the final outcome for one such run with $K=100$. The left panel shows the retained derivative files in the $(\sigma_a,\sigma_{s})$ plane. The grey markers denote the seeds whose snapshots are all discarded, while the colored markers denote the seeds that contribute retained files; the color indicates the number of retained files for each seed. We observe that the random seeds are not uniformly retained across the $(\sigma_a,\sigma_{s})$ space, but rather cluster in the regions with smaller $\sigma_a$ and $\sigma_{s}$. This is consistent with the intuition that, in the more diffusive regime with larger $\sigma_a$ and $\sigma_{s}$, the higher-order moments are smaller in magnitude and less spatially varying, which leads to lower scores and less retention. The right panel shows the ranked file scores across all snapshot files, with the score threshold induced by the storage budget $K=100$ shown as a dashed line. The files above the threshold are retained for training, while those below are discarded. This hard selection procedure ensures that the training data for the ML closure focuses on the most informative snapshots, which can lead to better learning efficiency and improved performance within the storage constraints. 
\begin{figure}[htbp]
    \centering
    \begin{subfigure}[b]{0.48\textwidth}
        \centering
        \includegraphics[width=\textwidth]{\figureprefix 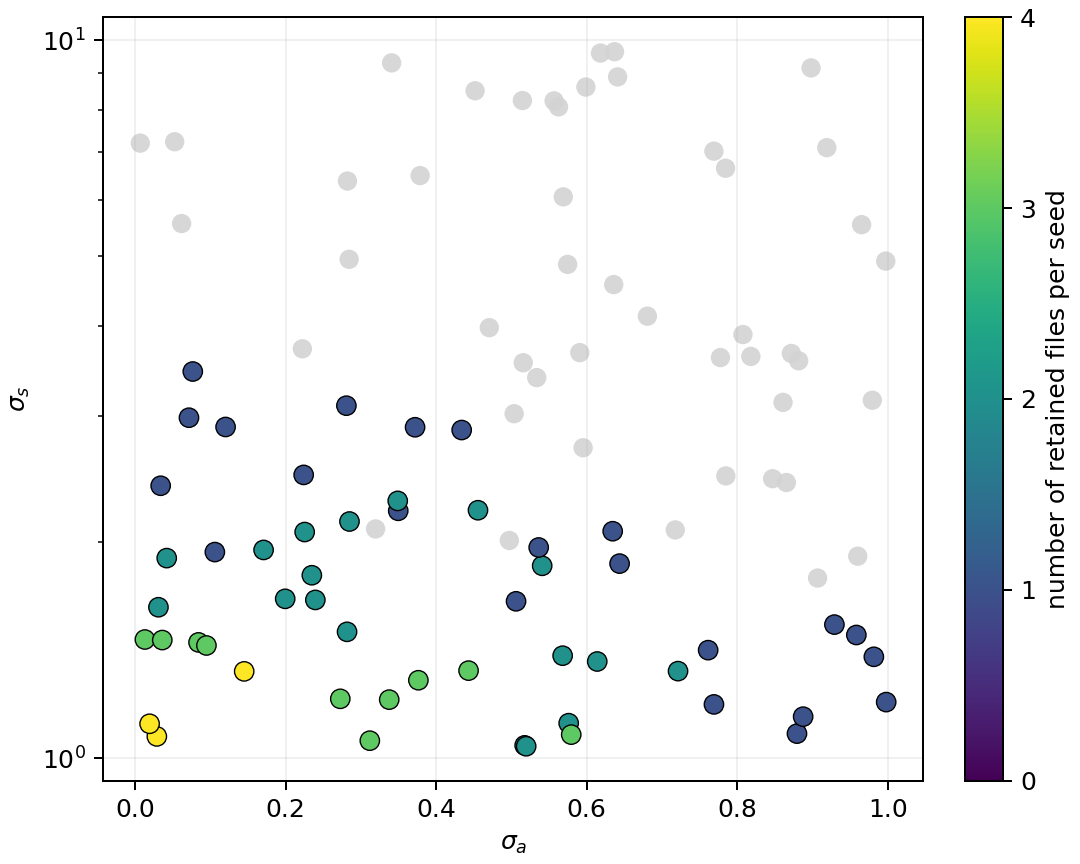}
        \caption{Sampled random seeds and retained training files in the $(\sigma_a,\sigma_s)$ space.}
    \end{subfigure}
    \hfill
    \begin{subfigure}[b]{0.48\textwidth}
        \centering
        \includegraphics[width=\textwidth]{\figureprefix 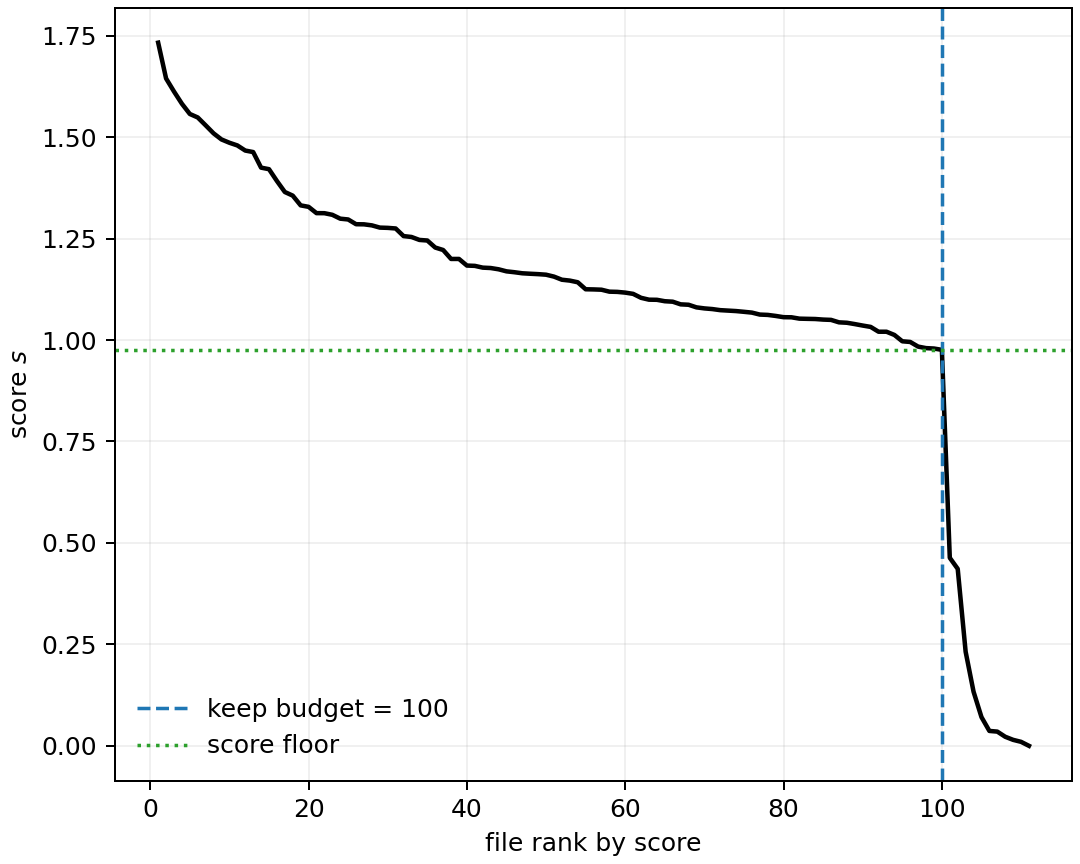}
        \caption{Ranked file score with keep budget and score floor.}
    \end{subfigure}
    \caption{Task 3: Left: All 100 randomly sampled seeds are shown in the $(\sigma_a,\sigma_s)$ plane. Gray markers denote that all snapshot files with the corresponding seed are not retained in the training set, while colored markers denote seeds that contribute retained training files; the color indicates the number of retained files for each seed. Right: The ranked stream scores of all snapshot files across 100 seeds, with the score threshold induced by the storage budget $K=100$ shown as a dashed line. The files above the threshold are retained for training, while those below are discarded.}
    \label{fig:multisine-hard-file-selection}
\end{figure}

After training the ML closure on the selected snapshot files, we test the trained closure on new trajectories with different random seeds and material parameters. In particular, we test the ML closure with $N=5$ on the trajectory with random seed $100$, which is outside the training set, and compare it against the linear $P_5$ baseline and the $P_{50}$ reference. The results are shown in Fig.~\ref{fig:multisine-p5-testing-contour-seed100}. We observe that, on this unseen trajectory with different material parameters, the ML closure still produces a much smoother solution that is closer to the $P_{50}$ reference than the linear $P_5$ model. This suggests that the ML closure can generalize to new initial conditions and new material parameters within the same family of multi-sine solutions, and it can still improve over the linear $P_N$ baseline in this setting.
\begin{figure}[htbp]
    \centering
    \begin{subfigure}[b]{0.67\textwidth}
        \centering
        \includegraphics[width=\textwidth]{\figureprefix 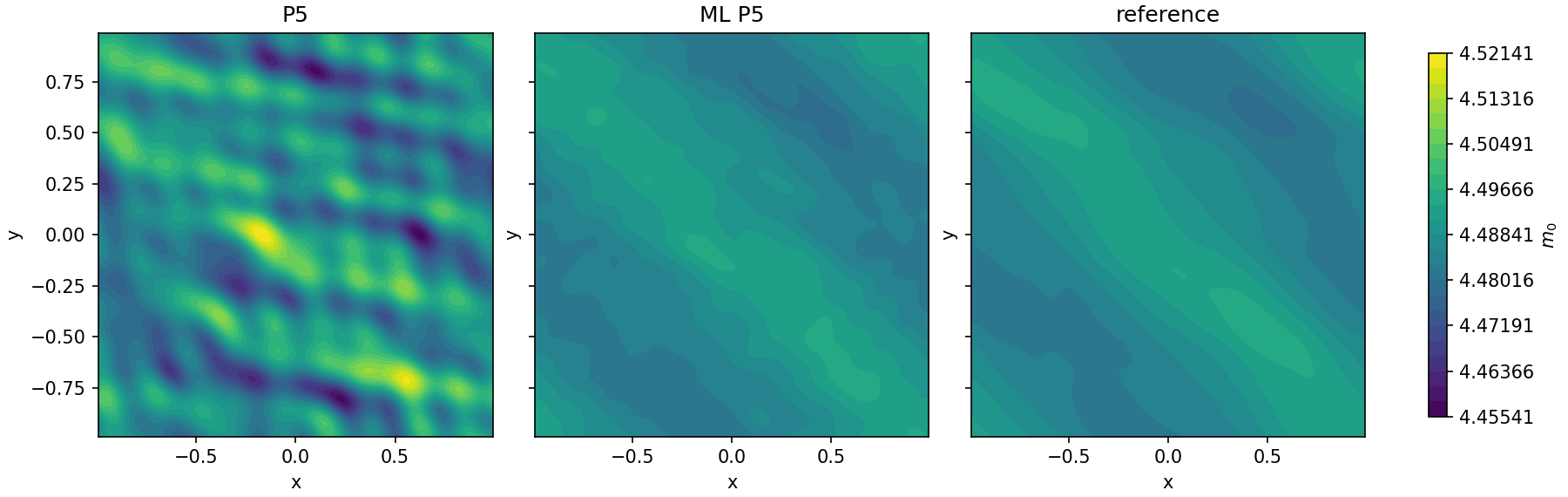}
        \caption{2D contour map}
    \end{subfigure}
    \hfill
    \begin{subfigure}[b]{0.32\textwidth}
        \centering
        \includegraphics[width=\textwidth]{\figureprefix 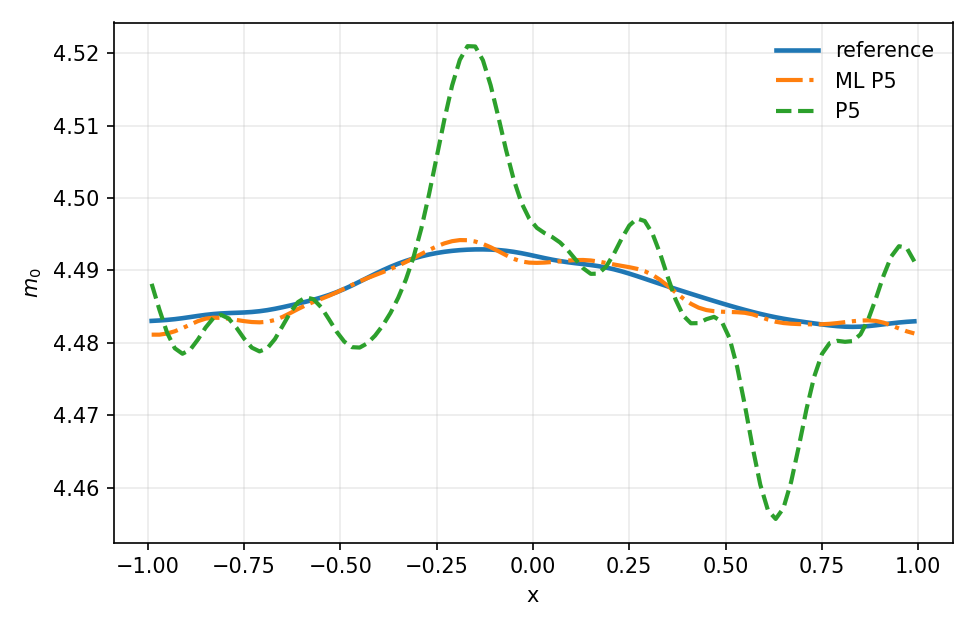}
        \caption{1D cut at $y=0$}
    \end{subfigure}
    \caption{Task 3: Comparison for the ML closure with $N=5$ and $P_5$ closure at $t=1$ and random seed $100$. The left figure shows the ML closure with $N=5$, the linear $P_5$ closure baseline, and the $P_{50}$ reference. The right figure shows the 1D cut at $y=0$. We observe that the ML solution is closer to the $P_{50}$ reference than the linear baseline in the full spatial field.}
    \label{fig:multisine-p5-testing-contour-seed100}
\end{figure}

\section{Conclusions}\label{sec:conclusion}

In this paper, we extended our previous work on ML closures for the RTE in 1D1V to the 2D2V setting. We introduced a block-diagonal symmetrizer and derived explicit algebraic conditions for the symmetrizer to ensure the hyperbolicity of the closed moment system. We then designed a new ML closure model that is guaranteed to satisfy these algebraic conditions by construction. We performed a series of numerical experiments to test the performance of the ML closure. The results show that the ML closure consistently improves over the linear $P_N$ model for a family of multi-sine initial conditions with fixed material parameters, and also generalizes to new trajectories within the same family. Future work includes testing the ML closure on a broader family of benchmark tests such as heterogeneous checkerboard media, and exploring more structural properties of the ML closure system such as rotational invariance and realizability-preserving properties. The ultimate goal is to design a closure model that can robustly improve over the linear $P_N$ model across a wide range of initial conditions and material regimes.

\section*{Acknowledgements}

The author thanks Andrew J. Christlieb at Michigan State University and Yingda Cheng at Virginia Tech for valuable discussions. 

\section*{Data availability}
No data was used for the research described in the article. The MATLAB and Python code for all numerical experiments is available from the author upon request.

\section*{CRediT authorship contribution statement}
J.H.: Conceptualization, Methodology, Software, Validation, Formal analysis, Investigation, Visualization, Writing - Original Draft, Writing - Review \& Editing.

\section*{Declaration of Generative AI and AI-assisted technologies in the
writing process}
During the preparation of this work, the author used ChatGPT to assist with numerical implementation, and drafting of the manuscript. After using this tool, the author reviewed and edited the content as needed and take full responsibility for the content of the publication.

\appendix

\section{Properties of Real Spherical Harmonics}

\begin{thm}[Real spherical harmonics form an orthonormal basis]\label{thm:real-sph-harm-orthonormal}
Let
\[
N_{l,m}:=\sqrt{\frac{2l+1}{4\pi}\frac{(l-m)!}{(l+m)!}},
\qquad l\ge 0,\quad 0\le m\le l,
\]
and define
\[
\Phi_l^0(\mu,\varphi):=N_{l,0}P_l(\mu),
\]
and, for \(1\le m\le l\),
\[
\Phi_l^{m,\mathrm c}(\mu,\varphi)
:=\sqrt{2}\,N_{l,m}P_l^m(\mu)\cos(m\varphi),
\qquad
\Phi_l^{m,\mathrm s}(\mu,\varphi)
:=\sqrt{2}\,N_{l,m}P_l^m(\mu)\sin(m\varphi).
\]
Then the family
\[
\bigl\{\Phi_l^0,\ \Phi_l^{m,\mathrm c},\ \Phi_l^{m,\mathrm s}
:\ l\ge 0,\ 1\le m\le l\bigr\}
\]
is an orthonormal basis of \(L^2(\mathbb S^2)\).
\end{thm}
\begin{proof}
    We first derive the integration formula for the surface integral on the sphere with the variables \((\mu,\varphi)\). With the spherical coordinates \((\theta,\varphi)\) with $\Omega = (\sin\theta\cos\varphi,\sin\theta\sin\varphi,\cos\theta)$, the surface integral of a function \(f:\mathbb S^2\to\mathbb R\) can be written as
    \begin{equation}
    \int_{\mathbb S^2} f(\Omega)\,d\Omega = \int_0^{2\pi}\int_0^\pi f(\theta,\varphi)\sin\theta\,d\theta\,d\varphi.
    \end{equation}
    Now with the change of variable $\mu=\cos\theta$, we have $d\mu=-\sin\theta\,d\theta$, so that
    \[
    \int_{\mathbb S^2} f(\Omega)\,d\Omega
    =
    \int_0^{2\pi}\int_0^\pi f(\theta,\varphi)\sin\theta\,d\theta\,d\varphi
    =
    \int_0^{2\pi}\int_{-1}^1 f(\mu,\varphi)\,d\mu\,d\varphi.
    \]

    We now check the orthonormality in the following cases: 
    \begin{enumerate}

    \item 
    The orthogonality of (\(\Phi_l^0\), \(\Phi_{l'}^{m,\mathrm c}\)), (\(\Phi_l^0\), \(\Phi_{l'}^{m,\mathrm s}\)), and (\(\Phi_{l'}^{m,\mathrm c}\), \(\Phi_{l'}^{m,\mathrm s}\)) can be easily checked by the orthogonality of the trigonometric functions of \(\varphi\).

    \item 
    The orthonormality of $\Phi_l^0$ and $\Phi_{l'}^0$ can be checked by directly computing
    \begin{equation}
    \begin{aligned}
        \int_{\mathbb S^2}\Phi_l^0\Phi_{l'}^0\,d\Omega ={}&
        2\pi N_{l,0}N_{l',0}\int_{-1}^1 P_l(\mu)P_{l'}(\mu)\,d\mu \\
        ={}& \frac{2\pi}{4\pi}\sqrt{(2l+1)(2l'+1)}\int_{-1}^1 P_l(\mu)P_{l'}(\mu)\,d\mu \\
        ={}& \frac{2\pi}{4\pi}\sqrt{(2l+1)(2l'+1)} \frac{2}{2l+1} \delta_{ll'} = \delta_{ll'}.
    \end{aligned}
    \end{equation}
    Here we use the orthogonality relation of the Legendre polynomials:
    \begin{equation}
        \int_{-1}^1 P_l(\mu)P_{l'}(\mu)\,d\mu = \frac{2}{2l+1}\delta_{ll'}.
    \end{equation}

    \item 
    The orthonormality of \(\Phi_l^{m,\mathrm c}\) and \(\Phi_{l'}^{m',\mathrm c}\) can be checked by directly computing
    \begin{equation}
    \begin{aligned}
        \int_{\mathbb S^2}\Phi_l^{m,\mathrm c}\Phi_{l'}^{m',\mathrm c}\,d\Omega ={}&
        2N_{l,m}N_{l',m'}\Bigl(\int_{-1}^1 P_l^m(\mu)P_{l'}^{m'}(\mu)\,d\mu\Bigr)\Bigl(\int_0^{2\pi}\cos(m\varphi)\cos(m'\varphi)\,d\varphi\Bigr) \\
        ={}& 2N_{l,m}N_{l',m'} \frac{2(l+m)!}{(2l+1)(l-m)!}\delta_{ll'} \pi \delta_{mm'} \\
        ={}& \delta_{ll'}\delta_{mm'}.
    \end{aligned}
    \end{equation}
    Here we use the orthogonality relation of the associated Legendre polynomials:
    \begin{equation}
        \int_{-1}^1 P_l^m(\mu)P_{l'}^m(\mu)\,d\mu = \frac{2(l+m)!}{(2l+1)(l-m)!}\delta_{ll'}
    \end{equation}
    Similar computation gives
    \begin{equation}
    \int_{\mathbb S^2}\Phi_l^{m,\mathrm s}\Phi_{l'}^{m',\mathrm s}\,d\Omega = \delta_{ll'}\delta_{mm'}.
    \end{equation}
    The proof is complete.
\end{enumerate}
\end{proof}

\end{document}